\documentclass[a4paper,11pt]{amsart}
\pdfoutput=1
\usepackage{latexsym, amssymb,amsthm,amsmath,color}
\usepackage{xspace}
\usepackage{tikz-cd}
\usepackage{mathtools}
\usepackage[all]{xy}
\usepackage[abs]{overpic}
\usepackage{hyperref}
\makeatletter
\renewcommand\subsubsection{\@startsection{subsubsection}{3}%
	\z@{.5\linespacing\@plus.7\linespacing}{-.5em}%
	{\normalfont\bfseries}}
\makeatother
\newcommand{\cc}{\ensuremath{\mathbb{C}}\xspace}

\newcommand{\rr}{\ensuremath{\mathbb{R}}\xspace}
\newcommand{\zz}{\ensuremath{\mathbb{Z}}\xspace}
\newcommand{\D}{\ensuremath{\mathbb{D}}\xspace}
\newcommand{\abs}[1]{\left| #1 \right|}

\newcommand{\set}[1]{\lbrace #1 \rbrace}

\newcommand{\rest}[2]{\left.{#1}\right|_{#2}}
\renewcommand{\bar}{\overline}
\renewcommand{\tilde}{\widetilde}

\newcommand{\inp}[1]{\left\langle #1 \right\rangle}

\DeclareMathOperator{\Res}{Res}

\newcommand{\oo}{\mathcal{O}}

\newcommand{\elli}{\mathcal{A}_{\abs{D}}}

\DeclarePairedDelimiter\floor{\lfloor}{\rfloor}
\renewcommand{\phi}{\varphi}
\usepackage{thmtools}
\declaretheorem[style=definition,qed=$\diamondsuit$]{definition}
\declaretheorem[style=definition,qed=$\triangle$,sibling=definition]{example}
\declaretheorem[style=plain,sibling=definition]{theorem}
\declaretheorem[style=plain,sibling=definition]{lemma}
\declaretheorem[style=plain,sibling=definition]{proposition}
\declaretheorem[style=plain,sibling=definition]{corollary}
\declaretheorem[style=definition,qed=$\diamondsuit$,sibling=example]{claim}
\declaretheorem[style=definition,qed=$\diamondsuit$,sibling=claim]{remark}

\newtheorem*{claim*}{Claim}
\numberwithin{theoremalpha}{section}

\numberwithin{equation}{section}
\numberwithin{definition}{section}
\numberwithin{theorem}{section}
\numberwithin{proposition}{section}
\numberwithin{lemma}{section}
\numberwithin{example}{section}
\numberwithin{remark}{section}
\numberwithin{corollary}{section}
\definecolor{ao(english)}{rgb}{0.0, 0.5, 0.0}
\newtheoremstyle{named}{}{}{\itshape}{}{\bfseries}{.}{.5em}{\thmnote{#3}#1}
\theoremstyle{named}
\newtheorem*{namedtheorem}{}
\numberwithin{equation}{section}
\begin{document}
\title{Fibrations in semi-toric and generalized complex geometry}
\author{Gil R. Cavalcanti}
\address{Department of Mathematics, Utrecht University, 3508 TA Utrecht, The Netherlands}
\email{g.r.cavalcanti@uu.nl}
\author{Ralph L.\ Klaasse}
\address{Department of Mathematics, KU Leuven, Celestijnenlaan 200B, BE-3001 Leuven, Belgium}
\email{r.l.klaasse@gmail.com}
\author{Aldo Witte}
\address{Department of Mathematics, Utrecht University, 3508 TA Utrecht, The Netherlands}
\email{g.a.witte@uu.nl}
\begin{abstract} 
This paper studies the interplay between self-crossing boundary Lefschetz fibrations and generalized complex structures. We show that these fibrations arise from the moment maps in semi-toric geometry and use them to construct self-crossing stable generalized complex four-manifolds using Gompf--Thurston methods for Lie algebroids. These results bring forth further structure on  several previously known examples of generalized complex manifolds. We moreover show that these fibrations are compatible with taking connected sums, and use this to prove a singularity trade result between two types of singularities occurring in these fibrations.
\end{abstract}
\maketitle
\tableofcontents
%
\section{Introduction}
\subsection{Fibrations in geometry}
A general theme in the world of geometric structures is the interplay between specific geometric structures and particular types of maps. The simplest example arises in the context of fibrations, where one studies whether the presence of a geometric structure on the base and the fibre implies its existence on the total space. However, to obtain interesting spaces and geometric structures, one should allow maps to have more singularities.

A concrete case that illustrates this comes from symplectic geometry: symplectic fibrations play an important role, but they are quite rare on general symplectic manifolds. Instead, if one allows the fibration to have Lefschetz singularities one obtains enough flexibility to establish broad existence results \cite{MR1802722,MR2108373,GS99}.

In this paper we use the point of view of decoupling the differential properties of maps from the desired underlying geometric structure. This decoupling has proven useful and allowed for several extensions of the results mentioned above, including those in \cite{AkbulutKarakurt08,MR2140998,MR2253445,Baykur08,Baykur09,MR3653062,CK18,CavalcantiKlaasse19,EtnyreFuller06,GayKirby07,Lekili09}. Here, allowing for maps to have Lefschetz and similar other singularities lead to existence results for several different types of geometric structures.

Another way in which singular fibrations arise is from proper group actions. In this setting the ``fibration'' condition translates to the group action being free, which is restrictive. The quotient map of a generic action will have singularities at points with non-trivial isotropy. Such singular fibrations are particularly well studied for torus actions, $T^n\times M^{2n} \to M^{2n}$, where, even if the actions considered are not free, they are well-behaved enough to ensure that the quotient space admits the structure of a manifold with corners. The coupling of torus actions with geometric structures leads to many fruitful concepts, one of the highlights being toric geometry.

The Lefschetz and the toric pictures come together in semi-toric geometry \cite{VP08}, where maps are allowed to have both types of singularities. In this setting the maps can have three types of singularities: elliptic, elliptic--elliptic and focus--focus, with the latter being equivalent to Lefschetz singularities. However, in semi-toric geometry the decoupling of maps and geometric structures has happened only partially, since local singular behaviour (Lefschetz and toric) and more global properties (integral affine structures) are mixed, leading to topological results \cite{MR2670163}.

In this paper we introduce the differential objects hinted at by semi-toric geometry: these are called \emph{self-crossing boundary (Lefschetz) fibrations} (c.f.\ Definitions~\ref{def:bf} and \ref{def:blf}). These types of maps incorporate local phenomena from both symplectic fibrations and quotient maps of semi-toric manifolds, but do not require global structures to be present, such as group actions or integral affine structures. In this paper we use these boundary fibrations to construct an a priori seemingly unrelated geometric structure, namely a generalized complex structure. 
\subsection{Generalized complex structures}
Generalized complex structures \cite{H03,Gua07} are a simultaneous generalisation of symplectic and complex structures. Infinitesimally these structures induce the product of a symplectic and complex vector space on each tangent space. However, the number of complex and symplectic directions, called the \emph{type}, can vary from point to point, leading to the notion of type change. These type-changing generalized complex structures are among the most interesting to study. Within the type-changing generalized complex structures, one class was put forward in \cite{CG17,CKW20} for being geometrically very rich and having well-controlled singular behaviour: \emph{self-crossing stable generalized complex structures}.

It is shown in \cite{CKW20} that self-crossing stable generalized complex structures on a manifold $M$ are in one-to-one correspondence with certain Lie algebroid symplectic structures, so that this paper makes extensive use of Lie algebroids. The singularities at the type-change locus $D$ induces a Lie algebroid $\mathcal{A}_{|D|} \to M$ called the \emph{self-crossing elliptic tangent bundle}, and the generalized complex structure makes it into a symplectic Lie algebroid, carrying an \emph{elliptic symplectic structure}. An elliptic symplectic structure corresponds to a self-crossing stable generalized complex structure if the locus $D$ is co-orientable and its so-called \emph{index} is $1$.

Given a self-crossing boundary Lefschetz fibration $f\colon (M,D) \to (N,Z)$ where $Z$ is a hypersurface of $N$, there is another relevant Lie algebroid, namely the \emph{self-crossing log-tangent bundle} $\mathcal{A}_Z \to N$. The map $f$ has singularities precisely such that it induces a Lie algebroid morphism
\[(\varphi,f)\colon (\mathcal{A}_{|D|},M) \to (\mathcal{A}_Z,N),\]
where $\varphi$ is now a Lie algebroid version of a Lefschetz fibration. The relevant geometric structure on the base of this fibration is a symplectic structure on $\mathcal{A}_Z$, also known as a \emph{self-crossing log-symplectic structure}. These have also appeared in \cite{GLPR17,MS18}. The Lie algebroid fibration is said to be \emph{compatible} with the elliptic symplectic structure on its total space if $\ker \varphi \subseteq \mathcal{A}_{|D|}$ consists of symplectic subpsaces. In turn, a stable generalized complex structure is \emph{compatible} with a boundary Lefschetz fibration if its induced elliptic symplectic structure is.
\subsection{Results}
In this section we describe the main results obtained in this paper. In the interest of brevity more precise versions of the below results can be found in the main body of the text.
\subsubsection*{Existence}
Following the strategy of constructing symplectic structures out of fibrations, we prove a Gompf--Thurston theorem for self-crossing stable generalized complex structures. This result is the generalisation of a similar result for stable generalized complex structures with embedded type-change locus appearing in \cite{CK18}, but requires several adaptations of the argument. We say that a map $f$ is \emph{homologically essential} if its generic fibre is non-trivial in homology.
\begin{namedtheorem}[Theorem~\ref{th:main}]
Let $f\colon (M^4,D^2) \rightarrow (N^2,\partial N)$ be a homologically essential self-crossing boundary Lefschetz fibration. Then $M^4$ admits an elliptic symplectic structure compatible with $f$, which induces a self-crossing stable generalized complex structure compatible with $f$ if the locus $D$ is co-orientable and its index is equal to $1$.
\end{namedtheorem}
\subsubsection*{Construction}
Having established that boundary Lefschetz fibrations supply self-crossing stable generalized complex structures, we decouple the map from the geometric structure and study them separately. These types of maps are flexible enough to admit connected sums:
\begin{namedtheorem}[Theorem~\ref{thm:glue}]
Let $f_i\colon (M_i^4,D_i^2) \rightarrow (N_i^2,\partial N_i)$ for $i = 1,2$ be boundary Lefschetz fibrations and let $p_i \in M_i$ be such that $q_i = f_i(p_i)$ are corners of the manifolds $N_i$. Then there exists a boundary Lefschetz fibration on their connected sum,
\begin{align*}
f_1 \# f_2\colon (M_1 \#_{p_1,p_2}M_2,D_1 \# D_2) \rightarrow (N_1\#_{q_1,q_2}N_2,\partial (N_1\# N_2)),
\end{align*}
which is compatible with the inclusion $M_i \backslash \set{p_i} \hookrightarrow M_1 \# M_2$. Moreover, the map $f_1 \# f_2$ is homologically essential if and only if $f_1$ and $f_2$ are.
\end{namedtheorem}
This result is in sheer contrast with the situation in toric geometry. There is no symplectic connected sum procedure, and most of the manifolds obtained using the above proposition will have no toric structure. This difference in rigidity between the generalized complex and toric worlds is already apparent on the base of these fibrations. Namely, for generalized complex structures the base carries a self-crossing log-symplectic structure, which is quite flexible. On the other hand, in toric geometry the base carries an integral affine structure, which is very rigid. In other words, although toric manifolds do not behave well with respect to connected sums, the underlying torus actions and abstract quotient maps do.
\subsubsection*{Singularity trades}
The nodal trade procedure in semi-toric geometry exchanges elliptic--elliptic singularities of the moment map for focus--focus singularities \cite{Zung03,MR815106} and vice-versa \cite{MR2670163}. These procedures rely heavily on the existence of a singular integral affine structure on the base. Following our general strategy, decoupling the geometric structure from the maps allows us to prove an abstract statement for boundary Lefschetz fibrations:
\begin{namedtheorem}[Theorem \ref{th:smoothing}]
Let $f\colon (M^4,D^2) \rightarrow (N^2,\partial N)$ be a boundary Lefschetz fibration, and let $p \in M$ be an elliptic--elliptic singularity. Then there exists a boundary Lefschetz fibration
\[{\tilde{f}\colon (M,\tilde{D}) \rightarrow (\tilde{N},\partial \tilde{N})}\]
agreeing with $f$ outside a neighbourhood of $p$, and such that the elliptic--elliptic singularity is traded for a Lefschetz singularity. The map $\tilde{f}$ is homologically essential if and only if $f$ is.
\end{namedtheorem}
The proof of this result, and its converse, Theorem~\ref{th:converse trade}, relies on the connected sum procedure and the existence of a particular boundary Lefschetz fibration on $S^4$.
\subsubsection*{Examples}
Using simple manifolds as building blocks, the connected sum procedure allows us to construct many examples of boundary Lefschetz fibrations (and consequently of self-crossing stable generalized complex structures) on the following manifolds:
\begin{namedtheorem}[Theorem \ref{th:examples}]
The manifolds in the following two families 
\begin{itemize}
	\item $X_{n,\ell} := \#n(S^2\times S^2)\#\ell(S^1\times S^3)$, with $n,\ell \in \mathbb{N}$;
	\item $Y_{n,m,\ell} := \#n \cc P^2 \#m \bar{\mathbb{C}P}^2\#\ell (S^1 \times S^3)$, with $n,m,\ell \in \mathbb{N}$,
\end{itemize}
admit homologically essential boundary fibrations whenever their Euler characteristic is non-negative. Therefore, each of these manifolds admits a compatible elliptic symplectic structure, which induces a self-crossing stable generalized complex structure if $1-b_1 + b_2^+$ is even.
\end{namedtheorem}
Combining this result with Theorem~\ref{th:smoothing} we conclude that the above manifolds moreover admit stable generalized complex structures with embedded degeneracy loci. These examples have already appeared in the literature \cite{CG09,Tor12,GH16, MR3177992} and the authors obtained them as well in \cite[Theorem 7.5]{CKW20}. However, the above result shows that the structures in these examples can be made compatible with boundary fibrations.
\subsection*{Organisation of the paper}
This paper is organised as follows. In Section~\ref{sec:divisors} we recall from \cite{CKW20} the notions of self-crossing divisors, and their associated Lie algebroids and symplectic structures. We also recall the definition of self-crossing stable generalized complex structures and that they are in one-to-one correspondence with particular self-crossing elliptic symplectic structures. In Section~\ref{sec:boundarymaps} we extend the notion of boundary (Lefshetz) fibration from \cite{CK18} to allow for self-crossing of the degeneracy locus. We moreover prove the Gompf--Thurston result, Theorem~\ref{th:main}. In Section~\ref{sec:connsum} we show that boundary Lefschetz fibrations allow for taking connected sums and prove Theorem~\ref{thm:glue}. In Section~\ref{sec:singtrades} we will prove the singularity trade results, namely Theorem~\ref{th:smoothing} and Theorem~\ref{th:converse trade}. Finally in Section~\ref{sec:examples} we show that torus actions give rise to boundary fibrations, and exhibit several examples, including Theorem~\ref{th:examples}.
\subsection*{Acknowledgements}
RK was supported by ERC consolidator grant 646649 ``SymplecticEinstein'', and by the FWO and FNRS under EOS Project No.~G0H4518N. AW was supported by the NWO through the UGC Graduate Programme. The authors would like to thank Marco Gualtieri, Conan Leung, Eduard Looijenga, Maarten Mol and Alvaro Pelayo for useful discussions, and Rafael Torres for pointing out the examples obtained in \cite{MR3177992}.
%
\section{Divisors, Lie algebroids and symplectic structures}\label{sec:divisors}
In this section we study geometric structures with specific singularities. To work with these singularities we will recall the concept of a divisor, and the particular cases of log and elliptic divisors which we will mainly use in this paper. Using these divisors we will recall the associated Lie algebroids and their Lie algebroid symplectic structures. Then we will introduce the objects which we want to construct in this paper, namely stable generalized complex structures. We will show that these structures correspond to certain Lie algebroid symplectic structures, which is how we will treat them in the rest of the paper.
\subsection{Divisors}
We will use an adaptation of the notion of a divisor from algebraic geometry to smooth manifolds in order to describe the singularities of geometric structures. We will only briefly go over the main concepts we need and refer to \cite{CK18,CG17,CKW20} for more information.
\begin{definition}
A \textbf{divisor} $(L,\sigma)$ consists of a real/complex line bundle $L \rightarrow M$ and a section $\sigma \in \Gamma(L)$ with nowhere dense zero set. Its \textbf{associated ideal} $I_\sigma$ is obtained as the image of the induced map $\sigma\colon \Gamma(L^*) \to \rr/\cc$. 
\end{definition}
The \textbf{product} of two divisors $(L_i,\sigma_i) \to M$ for $i = 1,2$ is the divisor $(L_1 \otimes L_2, \sigma_1 \otimes \sigma_2) \to M$ which has $I_{\sigma_1\otimes \sigma_2} = I_{\sigma_1} \cdot I_{\sigma_2}$ as associated ideal.
\begin{definition}
\begin{itemize}
\item A \textbf{morphism} between divisors $(L_i,\sigma_i) \to M_i$, for $i = 1,2$,  is a  smooth map $\varphi\colon M_1 \to M_2$ and a bundle isomorphism $\Phi\colon \varphi^*L_2 \to L_1$ such that $\Phi^*(\sigma_2) = \sigma_1$;
\item A \textbf{diffeomorphism} between divisors is a morphism for which the map $\varphi$ is a diffeomorphism;
\item Finally, two divisors over the same manifold, $(L_i,\sigma_i) \to M$, are \textbf{isomorphic} if there is a morphism $(\varphi,\Phi)$ for which $\varphi$ is the identity map.\qedhere
\end{itemize}
\end{definition}
The existence of a morphism $(\varphi,\Phi)\colon (L_1,\sigma_1)\to (L_2,\sigma_2)$ is equivalent to the existence of a smooth map $\varphi \colon M_1 \to M_2$ such that $\varphi^*I_{\sigma_2} = I_{\sigma_1}$. In particular, a divisor is determined up to isomorphism by its ideal, which allows us to use divisors and their associated ideals interchangeably. We will encounter several specific divisor types.
\begin{definition}
A \textbf{smooth real/complex log divisor} is a real/complex divisor $(L,\sigma)$ for which $\sigma$ vanishes transversely. 
\end{definition}
The \textbf{vanishing locus} $\sigma^{-1}(0)$ of a real log divisor has codimension one and is denoted by $Z$. The vanishing locus of a complex log divisor has codimension two and is denoted by $D$. By locally demanding a divisor to be a product of log divisors we obtain the following.
\begin{definition}\label{def:xlog}
A \textbf{self-crossing real/complex log divisor} on a manifold $M$ is a divisor $(L,\sigma)$, such that for every point $p\in M$ there exists a neighbourhood $U$ of $p$ such that 
\begin{equation*}
	I_{\sigma}(U) = I_1\cdot \ldots \cdot I_j,
\end{equation*}
where $I_1,\ldots,I_j$ are real/complex log ideals with transversely intersecting vanishing loci.
\end{definition}
A self-crossing real log divisor is determined up to isomorphism by its vanishing locus $Z$, as its ideal equals the ideal of functions vanishing on $Z$. In contrast, for a self-crossing complex log divisor, the subspace $D$ does not determine the divisor.
\begin{definition}[\cite{CG17}]
A \textbf{smooth elliptic divisor} is a divisor $(R,q)$ where $R$ is a real line bundle and $D = q^{-1}(0)$ is a codimension-two submanifold along which the normal Hessian of $q$ is definite.
\end{definition}
The normal Hessian of $q$ is the section $\text{Hess}(q) \in \Gamma(D;{\rm Sym}^2 N^*D \otimes R)$ containing the leading term of the Taylor expansion of $q$. Again by taking appropriate products we obtain the following.
\begin{definition}[\cite{CKW20}]\label{def:xelliptic}
A \textbf{self-crossing elliptic divisor} is a divisor $(R,q)$ on a manifold $M$, such that for every point $p\in M$ there exists a neighbourhood $U$ of $p$ such that
\begin{equation*}
	I_{q}(U) = I_1\cdot \ldots \cdot I_j,
\end{equation*}
where the $I_1,\ldots,I_j$ are smooth elliptic ideals with transversely intersecting vanishing loci.
\end{definition}
We mostly deal with self-crossing divisors in this paper, and we will often omit the prefix ``self-crossing''. Whenever we mean a smooth log or elliptic divisor we will explicitly stress this. 

The vanishing loci of both log and elliptic divisors are not embedded, but are stratified.
\begin{definition}
Let $I$ be a real/complex log or elliptic divisor on $M$ with vanishing locus $W$. The \textbf{intersection number} of a point $p \in M$ is the minimum of the integers $j$ from Definition~\ref{def:xlog} or Definition~\ref{def:xelliptic} over all neighbourhoods $U$ of $p$. The \textbf{intersection number of the divisor} is the maximum of the intersection numbers of all points $p \in M$. If $I$ has intersection number equal to $n$, the sets $W(j)$ of points of intersection number at least $j$ induce a filtration of $M$:
\begin{equation*}
	M = W(0) \supset W(1) = W \supset \cdots \supset W(n),
\end{equation*}
with induced \textbf{stratification} with strata $W[j]$ of points with intersection number precisely $j$.
\end{definition}
That this is a stratification follows readily from the normal forms of the divisors in Remark~\ref{rem:exampledivisors}. Also note that if a divisor $I$ has intersection number $n$ and $i \leq n$ is given, then the restriction $\rest{I}{M\backslash W(i+1)}$ defines a divisor with intersection number $i$.
\begin{remark}\label{rem:exampledivisors} There is a standard example for each of the divisor types described above.
\begin{itemize}
	\item The \textbf{standard real log divisor} with intersection number $j$ on $\rr^j \times \rr^m$ is defined using the coordinates $(x_1,\ldots, x_j,y_i)$ by the ideal $I_Z := \inp{x_1\cdot \ldots \cdot x_j}$;
	\item The \textbf{standard complex log divisor} with intersection number $j$ on $\cc^j \times \rr^m$ is defined using the coordinates $(z_1,\ldots, z_j,y_i)$ by the ideal $I_D := \inp{z_1\cdot \ldots\cdot z_j}$;
	\item The \textbf{standard elliptic divisor} with intersection number $j$ on $\rr^{2j} \times \rr^m$ is defined using the coordinates $(x_1,y_1,\ldots, x_j,y_j,u_i)$ by the ideal $I_{\abs{D}} := \inp{(x_1^2+y_1^2)\cdot \ldots \cdot (x_j^2+y_j^2)}$.
\end{itemize}
Each of these examples provides the local normal form for their associated divisor type.
\end{remark}
\begin{example}
Another important example for this paper is a manifold with corners $(M,\partial M)$. The boundary of a manifold with corners naturally defines a real log divisor.
\end{example}
Given a self-crossing complex log divisor $(L,\sigma)$, its \textbf{associated (self-crossing) elliptic divisor} is the real divisor $(R,q)$ defined by $(R,q) \otimes \underline{\cc} = (L,\sigma) \otimes (\bar{L},\bar{\sigma})$. 
\subsection{Lie algebroids and residue maps}\label{sec:liealg}
Each of the divisors introduced in the previous section gives rise to a corresponding Lie algebroid via the Serre--Swan theorem and the local normal forms contained in Remark~\ref{rem:exampledivisors}.
\begin{definition}
Let $I_Z$ be a real log divisor, $I_D$ a complex log divisor and $I_{\abs{D}}$ be an elliptic divisor. The vector fields preserving each of these ideals define Lie algebroids:
\begin{itemize}
	\item $\mathcal{A}_Z \to TM$, the \textbf{real log-tangent bundle} associated to $I_Z$;
	\item $\mathcal{A}_D \to T_\cc M$, the \textbf{complex log-tangent bundle} associated to $I_D$;
	\item $\mathcal{A}_{\abs{D}} \to TM$, the \textbf{elliptic tangent bundle} associated to $I_{\abs{D}}$.\qedhere
\end{itemize}
\end{definition}
\begin{remark}\label{rem:liealgebroidlocal} The above Lie algebroids can be described in the local coordinates of Remark~\ref{rem:exampledivisors}. Indeed, around a point of intersection number $j$ we have
\begin{itemize}
	\item $\Gamma(\mathcal{A}_Z) = \inp{x_1\partial_{x_1},\ldots,x_j\partial_{x_j},\partial_{y_i}}$;
	\item $\Gamma(\mathcal{A}_D) = \inp{z_1\partial_{z_1},\partial_{\bar{z}_1},\ldots,z_j\partial_{z_j},\partial_{\bar{z}_j},\partial_{y_i}}$;
	\item $\Gamma(\mathcal{A}_{\abs{D}}) = \inp{r_1\partial_{r_1},\partial_{\theta_i}, \ldots,r_j\partial_{r_j},\partial_{\theta_j},\partial_{u_i}}$.
\end{itemize}
In the latter case, we have $r_i\partial_{r_i} := x_i\partial_{x_i}+y_i\partial_{y_i}$ and $\partial_{\theta_i} := x_i\partial_{y_i}-y_i\partial_{x_i}$.
\end{remark}
When $I_D$ is a complex log divisor on $M$ and $I_{\abs{D}}$ is its associated elliptic divisor, there is a fibre product relation between the corresponding Lie algebroids as follows:
\begin{equation*}
	\mathcal{A}_{D} \times_{T_{\cc}M} \mathcal{A}_{\bar{D}} \cong \mathcal{A}_{\abs{D}} \otimes \cc.
\end{equation*}
This isomorphism provides an inclusion $\iota^*\colon \Omega^\bullet(\mathcal{A}_D) \to \Omega^\bullet_{\cc}(\mathcal{A}_{\abs{D}})$. 

We now turn to describing several of the residue maps carried by these Lie algebroids. Let $(I_{\abs{D}},\mathfrak{o})$ be a smooth elliptic divisor, together with a co-orientation of $D$. As explained in \cite{CG17}, the elliptic tangent bundle has an \textbf{elliptic} and \textbf{radial residue} map. These are maps of cochain complexes and they extract the coefficients of the singular generators. In the coordinates of Remark~\ref{rem:liealgebroidlocal} these are given by
\begin{equation}\label{eq:elliptic and radial}
\begin{aligned}
	\Res_q\colon \Omega^{\bullet}(\elli) \rightarrow \Omega^{\bullet-2}(D),& \qquad \Res_q(\alpha)= \iota^*_D(\iota_{r\partial_r\wedge\partial_{\theta}}\alpha),\\
	\Res_r\colon \Omega^{\bullet}(\elli) \rightarrow \Omega^{\bullet-1}(S^1ND),&\qquad \Res_r(\alpha) = \iota^*_D(\iota_{r\partial_r}\alpha),
\end{aligned}
\end{equation}
where $S^1ND \to D$ is the $S^1$-bundle associated to the co-orientation $\mathfrak{o}$ of $D$.

These residue maps can be extended to self-crossing elliptic divisors if we restrict our attention to the stratum $D[1]$. We say that a self-crossing elliptic divisor is \textbf{co-oriented} if the normal bundle $ND[1] \to D[1]$ is oriented.
\begin{definition}\label{def:zeroresform}
Let $(I_{\abs{D}},\mathfrak{o})$ be a co-oriented self-crossing elliptic divisor. The {\bf elliptic} and {\bf radial residues} of $\alpha \in \Omega^{\bullet}(\elli)$ are $\Res_q(\alpha) := \Res_q(\iota^*_{D[1]}\alpha)$ and $\Res_r(\alpha) := \Res_r(\iota^*_{D[1]}\alpha)$.
\end{definition}
In later constructions, the cohomology of the complex of forms with vanishing radial residue will play a role:
\begin{lemma}\label{lem:zerorescohom} Let $I_{\abs{D}}$ be a self-crossing elliptic divisor on a manifold $M$ and let $\Omega_{0,0}^{\bullet}(\elli) \subset \Omega^{\bullet}(\elli)$ be the subcomplex defined as the kernel of the map $\Res_r$. Then the inclusion map $i\colon M\backslash D \to M$ of the complement of $D$ induces a quasi-isomorphism $i^*\colon \Omega^{\bullet}_{0,0}(\elli) \rightarrow \Omega^{\bullet}(M\backslash D)$.
\end{lemma}
\begin{proof}
The argument uses the observation from \cite{Del71} that it suffices to show that $\iota^*$ induces an isomorphism on the level of sheaf cohomology. Below we will implicitly identify the sheaf $\Omega^{\bullet}(M\backslash D)$ with its push-forward $\iota_*(\Omega^{\bullet}(M\backslash D))$. For all points $p \in M\backslash D$ there exists a contractible open neighbourhood $U$ of $p$ disjoint from $D$. On this open $\elli = TM$ and $\Res_r \equiv 0$, and therefore $\iota_*$ is simply the identity. Let $j$ be any integer less than or equal to the intersection number of $D$, take $p \in D[j]$ and let $U$ be a contractible open around $p$ as in Remark~\ref{rem:liealgebroidlocal}. In those coordinates, $H^\bullet_{0,0}(U,\elli)$ is the free algebra generated by $\set{1,d\theta_1,\ldots,d\theta_j}$. By an elementary argument $U\backslash D$ is homotopic to $\mathbb{T}^j$, and using this homotopy $\iota^*$ takes the generators of $H^{\bullet}_{0,0}(\elli)$ to the generators of $H^{\bullet}(U\backslash D)$. Therefore we conclude that $\iota^*$ is a local isomorphism, and consequently a global isomorphism.
\end{proof}
We will need a few more residue maps for self-crossing elliptic divisors:
\begin{definition}[\cite{CKW20}]
Let $(I_{\abs{D}},\mathfrak{o})$ be a co-oriented elliptic divisor and let $\omega \in \Omega^2(\elli)$. Let $p \in D(k)$ with $k\geq 2$ and consider oriented coordinates as in Remark~\ref{rem:liealgebroidlocal}. We define
\begin{equation*}
	\Res_{r_ir_j}\omega(p) := \omega_p(r_i\partial_{r_i},r_j\partial_{r_j}),~ \Res_{r_i\theta_j}\omega(p) := \omega_p(r_i\partial_{r_i},\partial_{\theta_j}),~ \Res_{\theta_i\theta_j}\omega(p) := \omega_p(\partial_{\theta_i},\partial_{\theta_j}). \qedhere
\end{equation*}
These pointwise expressions depend on $\mathfrak{o}$ and the ordering of coordinates, but only up to sign.
\end{definition}
\subsection{Lie algebroid symplectic structures}
We will use the language of symplectic Lie algebroids to translate certain Poisson and generalized complex structures into simpler Lie algebroid objects. Given a Lie algebroid two-form $\omega \in \Omega^2(\mathcal{A})$, we say it is nondegenerate if $\omega^{\flat}\colon \mathcal{A} \rightarrow \mathcal{A}^*$ is an isomorphism.
\begin{definition} Let $I_Z$ and $I_{|D|}$ be real log and elliptic divisors on a given manifold $M$. Then:
\begin{itemize}
		\item A form $\omega \in \Omega^2(\mathcal{A}_Z)$ is \textbf{log-symplectic} if $d \omega = 0$ and it is nondegenerate;
		\item A form $\omega \in \Omega^2(\mathcal{A}_{|D|})$ is \textbf{elliptic symplectic} if $d\omega = 0$ and it is nondegenerate.\qedhere
\end{itemize}
\end{definition}
One can prove Darboux-type normal form theorems for symplectic Lie algebroids using a thorough understanding of their Lie algebroid cohomology, by a straightforward adaptation of the Moser lemma. However, in the above cases this cohomology is generally locally non-trivial, so that there is no unique local model. In dimension two we have the following:
\begin{lemma}\label{lem:logsympnormalform} Let $I_Z$ be a real log divisor on $\Sigma^2$ and let $\omega \in \Omega^2(\mathcal{A}_Z)$ be a log-symplectic form. For each point $p \in Z[2]$ there are coordinates $(x_1,x_2)$ centered at $p$ and $\lambda \in \rr$ such that
\begin{align*}
	\omega = \lambda d\log x_1 \wedge d\log x_2.
\end{align*}
\end{lemma}
Since in two dimensions every nowhere zero two-form is closed and nondegenerate we have the following source of examples of log-symplectic manifolds:
\begin{lemma}\label{lem:logsympexistence}
Let $\Sigma^2$ be a compact oriented surface with corners. Then $(\Sigma,I_{\partial \Sigma})$ admits a log-symplectic structure.
\end{lemma}
\begin{proof} The ideal $I_{\partial \Sigma}$ defines a real log divisor. Because $\Sigma$ is oriented, let $h \in C^{\infty}(M)$ be a defining function for $\partial M$, so that $I_{\partial \Sigma} = \langle h \rangle$, and let $\omega \in \Omega^2(\Sigma)$ be a volume form. Then $h^{-1}\omega \in \Omega^2(\mathcal{A}_{\partial M})$ is a nondegenerate log two-form that is closed for dimensional reasons.
\end{proof}
\subsection{Self-crossing stable generalized complex structures}
In this section we recall the notion of a self-crossing stable generalized complex structure \cite{CKW20}. This is a well-behaved class of generalized complex structures \cite{Gua07}, i.e.\ complex structures on the bundle $\mathbb{T}M:= TM \oplus T^*M$. We furthermore recall that they are equivalent to certain elliptic symplectic structures.
\begin{definition}
A \textbf{generalized complex structure} on a manifold $M$ is a pair $(\mathbb{J},H)$ where $H \in \Omega^3(M)$ is a closed three-form and $\mathbb{J}$ is an endomorphism of $\mathbb{T}M$ for which $\mathbb{J}^2 = -\rm{Id}$ and the $+i$-eigenbundle, $L \subset (\mathbb{T}M) \otimes \cc$, is involutive with respect to the Dorfman bracket:
\begin{equation*}
	[\![X+\xi,Y+\eta]\!]_H := [X,Y] + \mathcal{L}_X\eta - \iota_Yd\xi + \iota_X\iota_Y H, \qquad X+\xi, Y+\eta \in \Gamma(\mathbb{T}M).\qedhere
\end{equation*}
\end{definition}
Two generalized complex structures $(\mathbb{J},H)$ and  $(\mathbb{J}',H')$ are \textbf{gauge equivalent} if there exists $B \in \Omega^2(M)$ such that $H' = H + dB$ and, using the associated map $B^\flat\colon TM \to T^*M$, we have
\begin{equation*}
	\mathbb{J}' = \left(\begin{smallmatrix}
	1 & 0 \\
	B^{\flat} & 1
	\end{smallmatrix}
	\right)
	\mathbb{J}
	\left(
	\begin{smallmatrix}
	1 & 0 \\
	-B^{\flat} & 1
	\end{smallmatrix}
	\right).
\end{equation*}
\begin{lemma}[\cite{MR2861779}]
Let $\mathbb{J} = \left(\begin{smallmatrix}
	J & \pi^{\sharp}_{\mathbb{J}} \\
	\sigma^{\flat} & -J^*
	\end{smallmatrix}
	\right)$
be a generalized complex structure on $M$. Then $\pi_{\mathbb{J}} \in \mathfrak{X}^2(M)$ is a Poisson structure on $M$. Moreover, if $\mathbb{J}$ and $\mathbb{J}'$ are gauge equivalent, then $\pi_{\mathbb{J}} = \pi_{\mathbb{J}'}$.
\end{lemma}
Given an element $X + \xi \in \mathbb{T}_{\cc}M$, let $(X+\xi) \cdot \rho := \iota_X \rho + \xi \wedge \rho$ denote the Clifford action of $\mathbb{T}M$ on elements $\rho \in \wedge^{\bullet}T^*_{\cc}M$. A generalized complex structure $\mathbb{J}$ is alternatively characterised by its \textbf{canonical bundle} $K \subset \wedge^{\bullet}T^*_{\cc}M$ defined by the relation
\begin{equation*}
L = \set{u \in \mathbb{T}_{\cc}M : u \cdot K = 0}.
\end{equation*}
Its dual carries a natural section $s \in \Gamma(K^*)$, given by the map which sends $\rho \in \Gamma(K)$ to its degree zero part, and is called the \textbf{anticanonical section} of $\mathbb{J}$. The pair $(K^*,s)$, called the \textbf{anticanonical divisor}, can be used to define a specific class of generalized complex structures:
\begin{definition}[\cite{CKW20}] Let $M$ be a manifold and $H \in \Omega^3(M)$ closed. A generalized complex structure $\mathbb{J}$ on $(M,H)$ is \textbf{stable with self-crossings} if $(K^*,s)$ defines a self-crossing complex log divisor. 
\end{definition}
As before we will often simply call these structures ``stable'', and when their divisor is in fact smooth we will explicitly stress this.

If $\mathbb{J}$ is a stable generalized complex structure on $M$, one can show that $\pi_{\mathbb{J}}$ admits a non-degenerate lift to $\elli$, the elliptic tangent bundle with respect to the elliptic divisor induced by $(K^*,s)$. Inverting this non-degenerate lift results in an elliptic symplectic form $\omega \in \Omega^2(\elli)$. Under certain conditions this procedure can be reversed:
\begin{theorem}[\cite{CKW20}]\label{th:correspondence}
Let $M$ be a manifold. There is a one-to-one correspondence between gauge equivalence classes of stable generalized complex structures on $M$ and isomorphism classes of co-oriented elliptic divisors together with an elliptic symplectic form $\omega \in \Omega^2(\elli)$, satisfying
\begin{itemize}
\item $\Res_q(\omega) = 0$,
\item $\Res_{\theta_ir_j}(\omega) = \Res_{r_i\theta_j}(\omega)$,
\item $\Res_{r_ir_j}(\omega) = -\Res_{\theta_i\theta_j}(\omega)$.
\end{itemize} 
 Explicitly this map is given by
\[
\left\{
\begin{array}{c}
(\mathbb{J}, H)\colon\\
\mathbb{J} \mbox{  is a stable GCS}
\end{array}
\right\}
 \to
 \left\{
\begin{array}{c}
(I_{|D|}, \mathfrak{o}, \pi_{\mathbb{J}}^{-1})\colon\\
(I_{|D|}, \mathfrak{o}) \mbox{ is a co-oriented elliptic divisor and }\\
\pi_{\mathbb{J}}^{-1} \mbox{ is an elliptic symplectic form satisfying the above relations}
\end{array}
\right\}.
\]
Here $(I_{\abs{D}},\mathfrak{o})$ is the co-oriented elliptic divisor induced by the anti-canonical divisor.
\end{theorem}
In the above, the co-orientation $\mathfrak{o}$ is defined using the fact that the normal derivative of the anti-canonical section $s$ induces an isomorphism $d^{\nu}s\colon \rest{ND}{D[1]} \simeq \rest{K^*}{D[1]}$.
\begin{example}\label{ex:basicexample}
Consider the generalized complex structure on $\cc^2$ with canonical bundle generated by
\begin{align*}
	\rho = z_1z_2 + \tau dz_1 \wedge dz_2, \quad \tau \in \cc.
\end{align*}
The anticanonical section is given by $z_1z_2 \in \Gamma(\underline{\cc})$, and therefore $\rho$ defines a stable structure with elliptic ideal $\abs{z_1}^2\abs{z_2}^2$. The elliptic symplectic form induced by $\rho$ is
\begin{align*}
	\omega = \text{Im}(\tau)(d\log r_1 \wedge d\log r_2 - d\theta_1 \wedge d\theta_2) + \text{Re}(\tau)(d\log r_1 \wedge d\theta_2 + d\theta_1 \wedge  d\log r_2). 
\end{align*}
This structure provides the normal form for a four-dimensional stable generalized complex structure around a point in $D[2]$.
\end{example}
In this paper we will predominantly consider examples of stable generalized complex manifolds for which the local normal form has parameter $\tau = i \lambda$ for a real number $\lambda$, thus it is worth recalling the following definition:
\begin{definition}
Let $M^4$ be a four-dimensional manifold endowed with an elliptic divisor. We say $\omega \in \Omega^2(\elli)$ with zero elliptic residue has \textbf{imaginary parameter} at a point $p \in D[2]$ if
\begin{itemize}
\item $\abs{\Res_{r_1r_2}(\omega)(p)} = \abs{\Res_{\theta_1\theta_2}(\omega)(p)}$,
\item $\Res_{r_1\theta_2}\omega(p) = 0$,
\item $\Res_{r_2\theta_1}\omega(p) = 0$.
\end{itemize}
We say that $\omega$ has imaginary parameter if it has imaginary parameter at all points $p \in D[2]$.
\end{definition}
Recall that these residues are only well-defined up to sign, so that their absolute values are well-defined. Although elliptic symplectic forms with imaginary parameter seem very close to being induced by generalized complex structures, and in fact locally they are, due to possible orientation issues they might not globally correspond to a stable generalized complex structure. To see when this is the case, we need the following definition:
\begin{definition}
Let $M^4$ be an oriented manifold with a co-oriented elliptic divisor $(I_{\abs{D}},\mathfrak{o})$. Given $p \in D[2]$, let $(D_1,D_2)$ be two local embedded submanfolds for which $D = D_1 \cup D_2$ around $p$. The \textbf{intersection index} of $p$ is
\[
\varepsilon_p = \begin{cases}
 +1& \quad \mbox{ if the isomorphism } N_pD_1\oplus N_pD_2 \simeq T_pM\mbox{  is orientation-preserving};\\
-1& \quad\mbox{ otherwise}.
\end{cases}\]
The \textbf{parity} of a connected component $D'$ of $D$ is given by the product $\varepsilon_{D'} = \prod_{p \in D'[2]} \epsilon_p$.
\end{definition}
The parity of a connected component $D'$ of $D$ does not depend on the choice of co-orientation of $D$ and if $D'[2]$ has $n$ points, a change of orientation of $M$ changes the parity of $D'$ by $(-1)^n$. We extend the definition of \textbf{parity} to a smooth connected component $D'$ of $D$ by declaring its parity $\varepsilon_{D'}$ to be $+1$ if $D'$ is co-orientable, and $-1$ if it is not.

An elliptic symplectic form $\omega$ defines an orientation because $\omega^n$ is non-zero outside a codimension-two subset. Using this orientation, we can determine when an elliptic symplectic form induces a stable generalized complex structure:
\begin{corollary}[\cite{CKW20}]\label{cor:ellitocomplex}
Let $M^4$ be a manifold endowed with a co-orientable elliptic divisor $I_{\abs{D}}$, and let $\omega \in \Omega^2(\elli)$ be elliptic symplectic with zero elliptic residue and imaginary parameter. If the parity of all connected components of $D$ with respect to the orientation determined by $\omega$ is $1$, then there exists a co-orientation $\mathfrak{o}$ for $I_{\abs{D}}$ such that $(I_{\abs{D}},\mathfrak{o},\omega)$ induces an equivalence class of stable generalized complex structures.
\end{corollary}
This corollary gives us the following strategy: first construct elliptic symplectic structures with imaginary parameter, and then compute the parity of the connected components of the divisor. This is more convenient than using Theorem~\ref{th:correspondence} directly, as it separates the construction of the symplectic structure from the existence of a particular co-orientation of the divisor.
%
\section{Boundary maps and Lefschetz fibrations}\label{sec:boundarymaps}
The game we play next is to single out a class of maps that admits enough singularities to make them interesting, while also giving us enough control on the singular behaviour so that we can use these maps to perform geometric constructions. The main point of this section is to extend the notion of boundary Lefschetz fibration defined in \cite{CK18} for manifolds with smooth divisors to manifolds with self-crossing divisors. This extension allows for maps to have one extra type of singularity: elliptic--elliptic type. This change allows us to get a much better grasp on many generalized complex manifolds as those can be easily described as fibrations with elliptic--elliptic singularities.
\subsection{Boundary maps}
Our first step is to single out a very general class of maps which is compatible with the Lie algebroids introduced in Section~\ref{sec:liealg}. These are the \emph{boundary maps} which already illustrate how singular behaviour of maps can be coupled with Lie algebroids.

We start with some basic terminology. A \textbf{pair}, $(M,D)$,  is a manifold $M$ together with a (possibly) immersed submanifold $D \subseteq M$. A \textbf{map of pairs} $f\colon (M,D) \to (N,Z)$ is a map $f\colon M \to N$ such that $f(D) \subseteq Z$. A \textbf{strong map of pairs} furthermore satisfies $f^{-1}(Z) = D$. Finally, $(N,Z)$ is a \textbf{log pair} if the vanishing ideal $I_Z$ is a log divisor ideal on $N$.
\begin{definition}
Let $f\colon (M,D) \rightarrow (N,Z)$ be a strong map of pairs onto a real log pair. Then $f$ is a \textbf{boundary map} if $I_{\abs{D}}:= f^*I_Z$ defines an elliptic divisor ideal.
\end{definition}
\begin{example}\label{ex:basic}
The basic example to have in mind for boundary map is
\[f_1\colon (\cc^2,D) \to (\rr^2,Z), \qquad f_1(z_1,z_2) = (|z_1|^2,|z_2|^2),\]
where $D\subset \cc^2$ and $Z \subset \rr^2$ are the two coordinate axes.

There are other examples of boundary maps that we will eventually exclude by imposing further requirements, but which are also interesting to keep in mind for now:
\[f_2 \colon (\cc^2,D) \to (\rr,\{0\}), \qquad f_2(z_1,z_2) = |z_1|^2|z_2|^2,\]
where $D\subset \cc^2$ is again the two coordinate axes and 
\[f_3 \colon (S^2,\{p_N,p_S\}) \to (S^1,\{-1\}), \qquad f_3(x,y,z) = \exp(\pi i z),\]
where $p_N,p_S$ are the north and south poles of the unit sphere and we regard $S^1$ as the complex numbers of length $1$.
\end{example}
Notice that in the first two examples above, the image of the maps considered are manifolds with corners, and for all intents and purposes we could have considered them as maps into their image with the divisor being determined by the boundary. This is in line with the original idea behind log geometry (also known as $b$-geometry) developed by Mazzeo and Melrose \cite{MR1348401}. The third map shows that sometimes the image may be a genuine manifold (without boundary). Note that $f_3$ factors through the height map $\tilde{f_3} \colon (S^2, \set{p_N,p_S}) \rightarrow (I,\partial I)$, $\tilde{f_3}(x,y,z) = z$,
\begin{center}
\begin{tikzcd}
	&(I,\partial I) \ar[d,"\exp{(\pi i \,\cdot)}"] & \\
	(S^2,\{p_N,p_S\}) \ar[r,"f_3"] \ar[ru,"\tilde{f_3}"] & (S^1,\{-1\}),
\end{tikzcd}
\end{center}
which has image a manifold with boundary, and boundary as divisor. This is a specific example of a more general construction, namely that we can ``cut $N$ open along $Z$''. Next we will describe this procedure, which justifies the name {\it boundary map}.
\begin{lemma}\label{lem:boundaryfication}
Let $(N,Z)$ be a real log pair with $N$ a manifold without boundary. Then there is a manifold with corners $\tilde{N}$ and a map $p\colon (\tilde N,\partial \tilde N) \to (N,Z)$ such that
\begin{itemize}
\item $p$ is a map of divisors,
\item $p\colon  \tilde N\backslash\partial \tilde N \to N\backslash Z$ is a bijection,
\item every point $x \in \tilde{N}$ has a neighbourhood $U$ such that $p\colon U \to p(U)$ is a diffeomorphism. 
\end{itemize}
Further, if $p'\colon (N',\partial N') \to (N,Z)$ is another manifold and map satisfying the properties above then there is a unique diffeomorphism $\Psi\colon (N',\partial N') \to (\tilde N,\partial \tilde N)$ for which $p' = p \circ \Psi$.
\end{lemma}
\begin{proof}
We start with a local construction. Denoting by $\rr^n_\ell$ the manifold $\rr^n$ with divisor given by the hyperplanes determined by the equation $x_1 \cdots x_\ell =0$, we let $\tilde{\rr}^n_\ell$ be given by
\[\tilde{\rr}^n_\ell = \bigcup^\bullet_{K \in \{-1,1\}^\ell}\{(x_1,\dots,x_n) \in \rr^n\colon k_i x_i\geq 0, \mbox{ where } K = (k_1,\dots k_\ell)\},\]
and we let $p\colon \tilde{\rr}^n_\ell \to \rr^n_\ell$ be the natural inclusion: $p(x) = x$. Figure~\ref{fig:boundarification} shows this construction for $\rr^2$ with the two coordinate axes as its real log divisor.
\begin{figure}
\begin{overpic}[unit=1mm,scale=0.2]{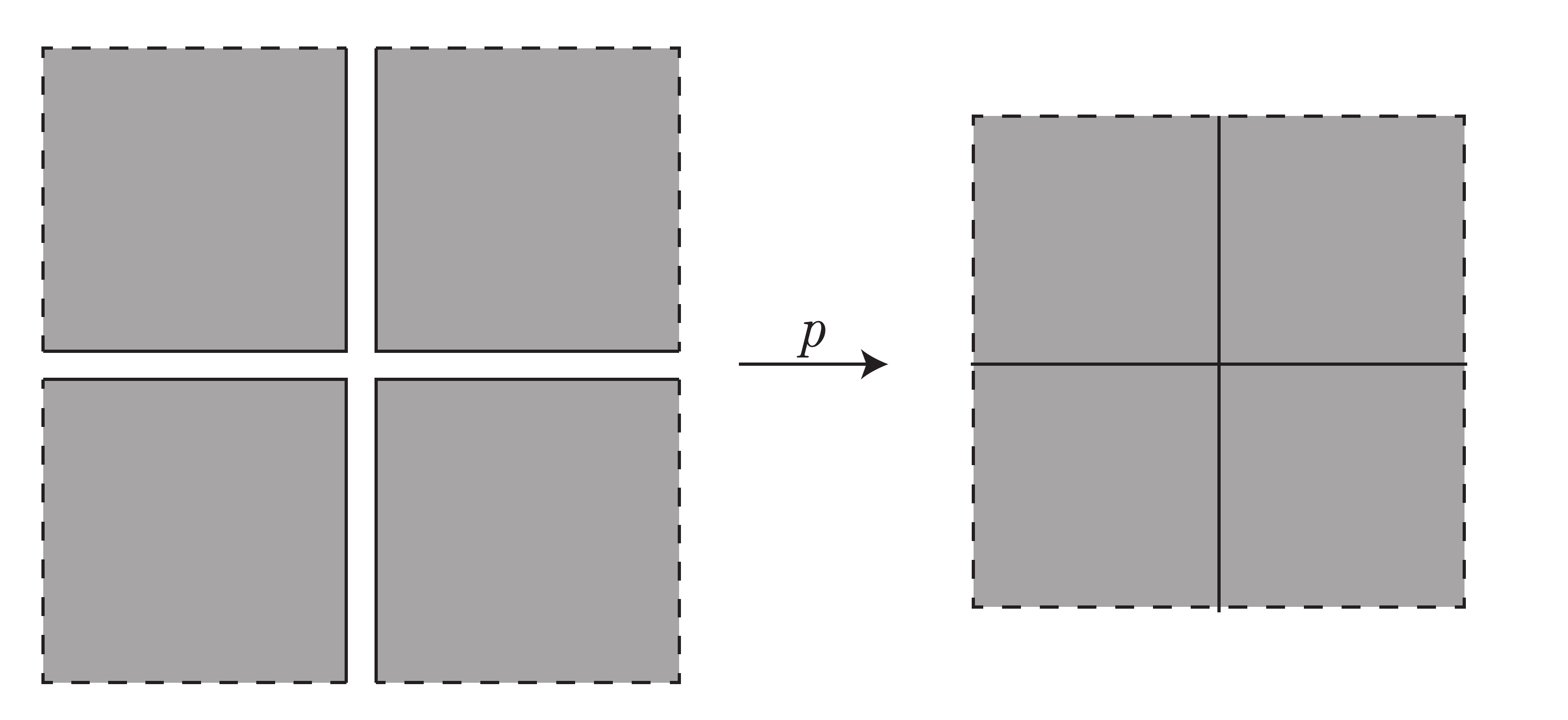}
	\put(14,-5){$\tilde{\rr}^2_2$}
	\put(52,-5){$\rr^2_2$}
\end{overpic}
\caption{Boundarification of $\rr^2$ with two coordinate axes.}\label{fig:boundarification}
\end{figure}
We call each connected component of $\tilde{\rr}^n_\ell$ defined above a {\it quadrant}.

Notice that $p\colon \tilde{\rr}^n_\ell \to \rr^n_\ell$ is a map of divisors, and if a smooth map $f\colon M \to \rr^n$ has its image in a quadrant, then it admits a smooth lift to $\tilde{\rr}^n_\ell$. Further, if $M$ is connected and the image of $f$ has points which are not in the hyperplanes determined by $x_1 \cdots x_\ell = 0$, then this lift is unique.

For the global construction, we observe that charts in $N$ provide a way to glue the local construction above to produce a manifold with corners. Indeed, given two charts that render the divisor in standard form, in their overlap the change of coordinates gives a diffeomorphism $\Phi \colon \rr^n_\ell \to \rr^n_\ell$, for some $\ell$. Since the charts are adapted to the divisor, $\Phi$ also induces a diffeomorphism of quadrants, that is, it lifts to a diffeomorphism $\tilde\Phi \colon (\tilde{\rr}^n_\ell,\partial \tilde{\rr}^{n}_\ell) \to (\tilde{\rr}^n_\ell,\partial \tilde{\rr}^{n}_\ell)$.

Since the changes of coordinates arising from an atlas for $N$ give rise to a \v{C}ech cocycle of diffeomorphisms, the same holds for their lifts, so the procedure can be used to produce a manifold with corners $\tilde{N}$. Further, the natural local maps ,``$p$'', defined in coordinate charts above patch together to give a global map of divisors $p\colon (\tilde N, \partial \tilde N) \to (N,Z)$. By construction, $p\colon \tilde N\backslash\partial \tilde N \to N \backslash Z$ is a bijection away from the divisors and a local diffeomorphism onto its local image.

Finally, if $p'\colon (N',\partial N') \to (N,Z)$ is a map of divisors with the two properties above then we show that $p'$ has a unique lift $\Psi\colon (N',\partial N') \to (\tilde N,\partial \tilde N) $:
\begin{center}
\begin{tikzcd}
&(\tilde{N},\partial \tilde N) \ar[d,"p"] & \\
(N',\partial N') \ar[r,"p'"] \ar[ru,"\Psi"] & (N,Z).
\end{tikzcd}
\end{center}
Indeed, in this case, $p^{-1} \circ p' \colon N'\backslash \partial N' \to \tilde N \backslash \partial \tilde N$ is a diffeomorphism and by the third property any point $x \in \partial N'$ has a connected neighbourhood $U \subset N'$ that maps diffeomorphically onto its image. Hence, taking $U$ small enough, since $U$ is connected, $p'(U)$ lies in a quadrant for a coordinate chart in $N$ and hence $p'$ has a unique (local) lift to $\tilde{N}$. Patching these local lifts together we obtain the map $\Psi$. Since $\Psi$ is a diffeomorphism in the interior of $N'$ and by construction also a local diffeomorphism for points in the boundary of $N'$ it is a global diffeomorphism.
\end{proof}
\begin{definition}
The {\bf boundarification} of a manifold without boundary together with a real divisor, $(N,Z)$, is a manifold with corners $(\tilde{N},\partial \tilde N)$ together with a map $p\colon (\tilde N,\partial \tilde N) \to (N,Z)$ satisfying the properties of Lemma~\ref{lem:boundaryfication}.
\end{definition}
\begin{example}
If we take $N$ to be the two-dimensional torus and $Z$ to be an embedded circle which represents a primitive homology class, the boundarification of $N$ is a cylinder and the map $p$ identifies the two ends of the cylinder. If we take $Z$ to be a pair of embedded circles intersecting transversely and which represent a basis for the homology of the torus, then the boundarification is a rectangle and the quotient map identifies opposite sides in the usual fashion.
\end{example}
\begin{proposition}\label{prop:boundarification of maps}
Let $f\colon (M,D)\to (N,Z)$ be a boundary map onto a manifold without boundary equipped with a real log divisor. Then there exists a unique boundary map $\tilde{f} \colon (M,D)\to (\tilde N,\partial \tilde N)$ to its boundarification that is a lift of $f$, i.e.\ which satisfies $f = p \circ \tilde{f}$ for $p \colon (\tilde N, \partial\tilde N) \to (N,Z)$.
\end{proposition}
\begin{proof}
All we need to prove is that every point $x\in M$ has a neighbourhood $U$ such that $f|_U\colon U \to N$ admits a unique lift $\tilde{f}|_U \colon U \to \tilde{N}$. Indeed, if this is the case, then any two such local lifts will agree in their overlap by uniqueness and hence the local lifts patch together to give a unique global map.

Because $D$ has codimension two in $M$ every point $x \in M$ has a neighbourhood $U$ such that $U \backslash D$ is connected, it follows that $f(U \backslash D)$ lies in a connected component of $N\backslash Z$. By taking $U$ small enough, we have that $f(U\backslash D)$ lies in a connected component of the complement of $Z$ in a coordinate patch $V \subset N$, that is, $f(U \backslash D)$ lies in a quadrant and, by continuity, so does $f(U)$. As such there is a unique lift to a map $\tilde f\colon U \to \tilde{N}$.
\end{proof}
We intend to use boundary maps to construct geometric structures on their total space. Thus we can, without loss of generality, assume that the target of a boundary map is $(N,\partial N)$, a manifold with corners whose real log divisor is determined by its boundary. This also explains the terminology ``boundary map''.
\subsection{Boundary Lefschetz fibrations}
The notion of a boundary map $f$ is still too general to give us enough information about the singularities of the map. To get a good grasp on $f$ we need to ensure that its singularities are well controlled and this is what we do next. There are two ways to constrain the singularities of $f$: we can either impose that they display a specific behaviour with respect to the ideals (and Lie algebroids) present, or we can impose that singularities disjoint from the vanishing loci of those ideals acquire a specific normal form. We will follow both routes here.

Note that a boundary map is by definition a map of pairs, so that it satisfies $f(D) \subseteq Z$. The first restriction we impose is that the map moreover respects the stratifications present on both $D$ and $Z$.
\begin{definition}
A \textbf{fibrating boundary map} is a boundary map $f\colon (M,D) \to (N,Z)$ such that for each $k \geq 1$ we have that:
\begin{itemize}
	\item $f\colon (M,D[k]) \to (N,Z[k])$ is a strong map of pairs;
	\item each restriction $\rest{f}{D[k]}\colon D[k] \rightarrow Z[k]$ is a submersion.\qedhere
\end{itemize}
\end{definition}
In Example~\ref{ex:basic}, $f_1$ and $f_3$ are fibrating boundary maps, while $f_2$ is not as it does not satisfy the first condition.

For a fibrating boundary map, $f$, we can use the ideals on $M$ and $N$ to control the singular behaviour of $f$ in a neighbourhood of their corresponding divisors. Concretely, we have a pointwise normal form for the map.
\begin{lemma}\label{lem:boundmaplocform}
Let $f\colon (M^n,D^{n-2}) \rightarrow (N^m,Z^{m-1})$ be a fibrating boundary map and let $x \in D[k]$. Then there exist coordinates $(x_1,\ldots,x_n)$ around $x$ and $(z_1,\ldots,z_k,y_i)$ around $f(x)$ such that
\begin{itemize}
	\item $Z$ is the standard log divisor with intersection number $k$ on $\rr^k \times \rr^{m-k}$
	\item $D$ is the standard elliptic divisor with intersection number $k$ on $\rr^{2k}\times \rr^{n-2k}$, and
	\item in these coordinates, the map $f$ takes the form
	\begin{align*}
		f(x_1,\ldots,x_n) = (x_1^2+x_2^2,\ldots,x_{2k-1}^2+x_{2k}^2,x_{n-m+k},\ldots,x_n).
	\end{align*}
\end{itemize}
Conversely, if for every point in $D$ the map $f$ is given in standard coordinates for the divisors by the expression above, then it is a fibrating boundary map.
\end{lemma}
\begin{proof}
Choose a tubular neighbourhood $\mathcal{V}$ of $Z[k]$ and denote by $\text{pr}_{Z[k]}\colon NZ[k] \rightarrow Z[k]$ the projection. Let $V' \subset \mathcal{V}$ be an open neighbourhood of $f(x)$ on which $Z[k]$ is the standard log divisor and write $V'\cap Z[k] = \set{z_1\cdot \ldots \cdot z_k = 0}$. Choose a coordinate system $(y_{k+1},\ldots,y_m)$ on $V' \cap Z[k]$, so that the set $\set{z_1,\ldots,z_k,\text{pr}^*_{Z[k]} y_{k+1},\ldots,\text{pr}^*_{Z[k]}y_m}$ forms a coordinate system on $V'$ which is possible because $\text{pr}_{Z[k]}$ is a submersion. Because $f$ is a morphism of divisors, $f^*(z_1 \cdot \ldots \cdot z_k)$ generates an elliptic divisor ideal on $U := f^{-1}(V') \subseteq M$. Using that $f$ is fibrating, after possibly shrinking $U$ around $x$, let $(x_1,\ldots,x_n)$ be coordinates on $U$ in which this is the standard elliptic divisor, and such that $f^*(z_j) = x_{2j-1}^2+x_{2j}^2$. Because the restriction $\rest{f}{Z[k]}$ is a submersion we see that
\begin{align*}
	\set{x_1,\ldots,x_{2k},f^*\text{pr}^*_{Z[k]}y_{k+1},\ldots,f^*\text{pr}^*_{Z[k]}y_m}
\end{align*}
forms a functionally independent set. We can complete this to a coordinate system on $M$ and relabel these as $(x_1,\dots,x_n)$. If we use the coordinate system $(z_1,\ldots,z_k,\text{pr}^*_{Z[k]}y_{k+1},\ldots,\text{pr}^*_{Z[k]}y_m)$ on $N$ and the above coordinates on $M$, then $f$ takes the required form.

The converse follows immediately from the local expression for $f$.
\end{proof}
\begin{remark}\label{rem:there will be order}
Even if $M$ and $N$ are oriented manifolds and we require the use of coordinate charts compatible with orientations, we can still arrange that the local expression for $f$ is given by the expression in Lemma~\ref{lem:boundmaplocform}. Indeed, using complex conjugation on the domain and permutation of the coordinates on both domain and codomain we can change a coordinate chart which is not compatible with the given orientations into one that is.

In four dimensions, if $D[2]$ is nonempty, Lemma~\ref{lem:boundmaplocform} implies that $N$ is two-dimensional. Moreover, in a neighbourhood of a point $p\in D[2]$, orientations of $M$ and $N$ in fact dictate which one is ``the first" strand of $D$ arriving at $p$ and which one is ``the second", as this information is determined by the orientation of $N$.
\end{remark}
\begin{lemma}\label{lem:confibresoverD}
Let $f\colon (M,D) \rightarrow (N,Z=\partial N)$ be a fibrating boundary map whose fibres near $D$ are connected. Then the fibres of $\rest{f}{D[k]}\colon D[k] \rightarrow Z[k]$ are connected for all $k\geq 1$.
\end{lemma}
\begin{proof}
The proof goes by induction over the strata. Note that $Z[k+1]$ is a hypersurface in $Z[k]$, and therefore
\begin{align*}
	\rest{f}{D[k]}\colon (D[k],D[k+1]) \rightarrow (Z[k],Z[k+1])
\end{align*}
is a fibrating boundary map for all $k \geq 0$. Applying \cite[Proposition 5.25]{CK18} to $\rest{f}{M \backslash D(2)}$ tells us that the fibres of $\rest{f}{D[1]}$ are connected. Thus we can apply the same result to $\rest{f}{D[1]}$ to conclude that the fibres of $\rest{f}{D[2]}$ are connected. Continuing inductively we arrive at the desired result.
\end{proof}
The conditions imposed on the maps have, up to this point, been on behaviour near $D$. Next we impose the conditions away from $D$:
\begin{definition}\label{def:bf}
A \textbf{boundary fibration} is a fibrating boundary map $f\colon (M,D) \to (N,Z)$ such that $\rest{f}{M\backslash D}\colon M\backslash D \to N \backslash Z$ is a surjective submersion.
\end{definition}
\begin{definition}\label{def:blf}
A \textbf{boundary Lefschetz fibration} is a fibrating boundary map $f\colon (M^{2n},D) \to (\Sigma^2,Z)$ between oriented manifolds such that $\rest{f}{M\backslash D}\colon M\backslash D \to \Sigma \backslash Z$ is a Lefschetz fibration. That is, the map $f\colon M \to N$ is proper, $f|_{M\backslash D}$ is injective on critical points and for each critical point $p \in M\backslash D$ there exist orientation-preserving complex coordinate charts centered at $p$ and $f(p)$ in which $f$ takes the form
\begin{equation*}
	f\colon \cc^n \to \cc, \qquad f(z_1,\ldots,z_n) = z_1^2 + \ldots + z_n^2.\qedhere
\end{equation*}
\end{definition}
If $M$ is four dimensional, the definition above allows for three different types of singularities. It is worth giving them names:
\begin{definition}
Let $f\colon M^4 \to \Sigma^2$ be a smooth map.
\begin{itemize}
\item An \textbf{elliptic singularity} of $f$ is a point $p$ for which $f$ has the local expression
\[f(x_1,x_2,x_3,x_4) = (x_1^2+x_2^2,x_4), \qquad x_i \in \rr;\]
\item An \textbf{elliptic--elliptic singularity} of $f$ is a point $p$ for which $f$ has the local expression
\[f(x_1,x_2,x_3,x_4) = (x_1^2+x_2^2,x_3^2+x_4^2),\qquad x_i \in \rr;\]
\item A \textbf{Lefschetz singularity} of $f$ is a point $p$ for which $f$ has the local expression
\[f(z_1,z_2) = z_1^2+z_2^2, \qquad z_i \in \cc.\qedhere\]
\end{itemize}
The level sets associated to these singularities are, respectively, an \textbf{elliptic, elliptic--elliptic and Lefschetz fibre}.
\end{definition}
The first two singularities above happen at the different strata of the divisor whereas the Lefschetz singularities do not interact with the divisor. In dimension four the geometry of these fibrations can be understood.
\begin{proposition}
Let $f\colon (M^4,D^2) \rightarrow (\Sigma^2,Z)$ be a boundary Lefschetz fibration with connected fibres, and let $D'$ be a connected component of $D$. Then
\begin{itemize}
\item The generic fibres of $f$ near $D$ are tori;
\item When $D'[2]$ has $k \geq 1$ points, then $D'$ is a union of $k$ pairwise transversely intersecting spheres.
\end{itemize}
In particular, if $D'[2] \neq \emptyset$, then $D'$ is co-orientable.
\end{proposition}
\begin{proof}
The first point follows immediately from \cite[Corollary 5.18]{CK18}.

For the second, assume that $D'[2]$ has at least one point. Then $\rest{f}{D[1]}\colon D[1] \rightarrow Z[1]$ is a surjective submersion by assumption, which by Lemma~\ref{lem:confibresoverD} has connected fibres. The corresponding locus $Z'[1]$ is a disjoint union of $k$ open intervals, and as the fibres of $\rest{f}{D[1]}$ are connected they must be circles. Therefore $D'[1]$ has to be a disjoint union of cylinders. The immersed submanifold $D'$ is obtained from $D'[1]$ by replacing the boundary circles by points and pairwise glueing these points, which implies it is as described above.

Finally, because each component of $D'$ is an immersed sphere and thus automatically co-orientable, each component of $D'[1]$ is also co-orientable.
\end{proof}
Therefore, to construct stable generalized complex structures using Corollary~\ref{cor:ellitocomplex}, the condition of co-orientability of $D$ is satisfied as long as $D'[2]$ is nonempty for every component $D'$ of $D$. For smooth components $D'$ of $D$ however, i.e.\ when $D'[2] = \emptyset$, co-orientability is not guaranteed.

\subsection{Boundary maps and Lie algebroids}
Given that the ideals $I_Z$ and $I_{|D|}$ of a boundary map determine Lie algebroids, one should expect that boundary maps (and their further specializations) are compatible with them. This is indeed the case. 
\begin{lemma}\label{lem:boundarymapmorphism}
Let $f\colon (M,D) \rightarrow (N,Z)$ be a boundary map. Then there is a Lie algebroid morphism $(\varphi,f)\colon \elli \rightarrow \mathcal{A}_Z$ such that $\varphi \equiv df$ on sections.
\end{lemma}
\begin{proof}
To prove that $df$ induces a Lie algebroid morphism $\varphi$, by \cite[Proposition 3.14]{CK18} it suffices to show that $f^*$ extends to an algebra morphism $\varphi^*\colon \Omega^{\bullet}(\mathcal{A}_Z) \rightarrow \Omega^{\bullet}(\elli)$. This can be done locally, so given $p \in D$ and $f(p) \in Z$ consider coordinates adapted to the divisors as in Remark~\ref{rem:exampledivisors}:
\begin{equation*}
	(X_1,Y_1,\ldots,X_s,Y_s,X_{2s+1},\ldots,X_n) \text{ around } p, \qquad (x_1,\ldots,x_j,y_{j+1},\ldots,y_{m}) \text{ around } f(p).
\end{equation*}
In these coordinates we have
\begin{align*}
	\Omega^{\bullet}(\elli) &= \inp{d\log r_1,d\theta_1,\ldots,d\log r_s,d\theta_s,dX_{2s+1},\ldots, dX_n},\\
	\Omega^{\bullet}(\mathcal{A}_Z) &= \inp{d\log x_1,\ldots,d\log x_j,dy_{j+1},\ldots,dy_m}.
\end{align*}
We must verify that $f^*(d \log x_i)$ defines an elliptic form. Because $f$ is a morphism of divisors and the ideals are locally principal, it sends generators to generators thus there must exist a nowhere-vanishing function $g$ such that $f^*(x_1\cdot\ldots\cdot x_j) = gr_1^2\cdot\ldots\cdot r_s^2$. Consequently, by functional indivisibility of the $r_i^2$ we conclude that $f^*(x_i) = h r_{i_1}^2\cdot\ldots\cdot r_{i_\ell}^2$ for some nowhere vanishing function $h$ and (possibly empty) subset $\{i_1,\ldots, i_\ell\} \subseteq \{1,\ldots,s\}$. We find that
\begin{equation*}
	f^*(d\log x_i) = d\log f^*(x_i) = d\log h + 2d\log r_{i_1}+ \ldots +2d\log r_{i_l},
\end{equation*}
which is an elliptic form as desired, so that $\varphi$ is a Lie algebroid morphism.
\end{proof}
The conditions imposed on boundary maps have a direct counterpart in Lie algebroid language. Given a Lie algebroid $\rho_\mathcal{A}\colon \mathcal{A} \to M$, let $M_\mathcal{A}$ be the open subset where the anchor map is an isomorphism.
\begin{definition}[\cite{CK18}] A Lie algebroid morphism $(\varphi,f)\colon (\mathcal{A},M) \to (\mathcal{A}',N)$ is said to be a:
\begin{itemize}
	\item \textbf{Lie algebroid fibration} if the induced morphism $\varphi^{!}\colon \mathcal{A} \to f^* \mathcal{A}'$ is surjective;
	\item \textbf{Lie algebroid Lefschetz fibration} if $M_{\mathcal{A}}$ is dense, $f^{-1}(N_{\mathcal{A}'}) = M_{\mathcal{A}}$ and there exists a discrete set $\Delta \subset M_{\mathcal{A}}$ such that
	\begin{itemize}
		\item $f|_{M_\mathcal{A}}\colon M_{\mathcal{A}} \to N_{\mathcal{A}'}$ is a Lefschetz fibration with ${\rm Crit}(f|_{M_\mathcal{A}}) = \Delta$;
		\item $(\varphi,f)\colon (\mathcal{A}, M\backslash f^{-1}(f(\Delta))) \to (\mathcal{A}',N\backslash f(\Delta))$ is a Lie algebroid fibration.
	\end{itemize}
	Note that the Lefschetz condition forces that ${\rm rank}(\mathcal{A}) = 2n$ and ${\rm rank}(\mathcal{A}') = 2$.\qedhere
\end{itemize}
\end{definition}
The following lemmas follow immediately from the definition, combined with Lemma~\ref{lem:boundarymapmorphism}.
\begin{lemma}\label{lem:algfib}
Let $f\colon (M,D) \rightarrow (N,Z)$ be a boundary fibration. Then there is a Lie algebroid fibration $(\varphi,f)\colon (\mathcal{A}_{\abs{D}},M) \rightarrow (\mathcal{A}_Z,N)$ such that $\varphi \equiv df$ on sections.
\end{lemma}
\begin{lemma}\label{lem:alglffib}
Let $f\colon (M^4,D^2) \rightarrow (N^2,Z^1)$ be a boundary Lefschetz fibration. Then there is a Lie algebroid Lefschetz fibration $(\varphi,f)\colon (\mathcal{A}_{\abs{D}},M^4) \rightarrow (\mathcal{A}_Z,N^2)$ such that $\varphi \equiv df$ on sections.
\end{lemma}
We summarise these statements and the relationship between the different concepts in the table below:\\
\begin{center}
\begin{tabular}{|c c c|}
\hline
Boundary (Lefschetz) fibration &$\Rightarrow$& Lie algebroid (Lefschetz) fibration\\
$\Downarrow$& & $\Downarrow$\\
Fibrating boundary map &$\Rightarrow$& Lie algebroid map {\it submersive} over the singular locus\\
$\Downarrow$& & $\Downarrow$\\
Boundary map &$\Rightarrow$& Lie algebroid map\\
\hline
\end{tabular}\\[24pt]
\end{center}
\subsection{Construction of self-crossing stable generalized complex structures}
With the desired notion of boundary Lefschetz fibration in hand, we are set to prove our first result relating them to stable generalized complex structures.  

From now on we will adopt the following convention: given a boundary Lefschetz fibration $f \colon (M,D) \to (N,Z)$, we will orient the fibres of $f\colon M\backslash D \to N\backslash Z$ by declaring that the orientation of the fibre together with the orientation of the base yield the orientation of $M$, so that each fibre determines a homology class on $M\backslash D$. With this convention, integration over the fibre is a well-defined operation which induces the natural pairing between homology and cohomology.
\begin{definition} A boundary Lefschetz fibration, $f\colon(M^4,D) \to (N^2,Z)$, is \textbf{homologically essential} if the homology class $[F]$ of a fibre of $f\colon M \backslash D \rightarrow N \backslash Z$ is non-trivial in $H_2(M\backslash D;\rr)$ or, equivalently, if there is a class $c \in H^2(M\backslash D;\rr)$ such that $\langle c,[F]\rangle \neq 0$.
\end{definition}
\begin{definition} A boundary Lefschetz fibration, $f\colon(M^4,D) \to (N^2,Z)$, and an elliptic symplectic form $\omega \in \Omega^2(\mathcal{A}_{|D|})$ are \textbf{compatible} if $\ker \varphi \subseteq \mathcal{A}_{|D|}$ consists of symplectic vector spaces, where $\varphi\colon \mathcal{A}_{|D|} \to \mathcal{A}_Z$ is the induced map of Lie algebroids.
\end{definition}
In what follows we will have two ongoing simplifying assumptions:
\begin{enumerate}
\item We will assume that the target manifold is $(N,\partial N)$. This is not a restriction since by Proposition~\ref{prop:boundarification of maps} we can lift $f$ to a boundary Lefschetz fibration over the boundarification of $(N,Z)$;
\item We will assume that the level sets of $f$ are connected. This also is not restriction since by \cite[Proposition 5.24]{CK18} we may assume that the generic fibres of $f$ are connected and Lemma~\ref{lem:confibresoverD} then implies that the level sets over $Z[1]$ and $Z[2]$ are connected as well.
\end{enumerate}
Before we continue it is worth to stop and take stock of where we stand and place our quest into context. The case when the elliptic divisor is smooth was already treated in \cite{CK18}. Even though there the authors only dealt with the compact case, the following is an immediate generalisation for a proper map:
\begin{theorem}[\cite{CK18}, Theorem 7.1]\label{th:wishful}
Let $(M^4,I_{|D|})$ be an oriented manifold with a smooth elliptic divisor and let $f\colon (M,D) \rightarrow (N^2,Z,\omega_{N})$ be a homologically essential, proper, Lefschetz fibration with connected fibres over a possibly open log-symplectic surface. Denote by $\phi\colon \mathcal{A}_{|D|} \to \mathcal{A}_Z$ the induced map of Lie algebroids. Let $c \in H^2(M\backslash D) = H^2_{0,0}(M,\mathcal{A}_{|D|})$ be a cohomology class such that $\langle c, [F]\rangle > 0$, where $F$ is a regular fibre of $f$. Then there exists a closed two-form $\eta \in \Omega^2_{0,0}(M,\mathcal{A}_{|D|})$ with $[\eta] =c$ and a positive function $\rho_0 \in C^\infty(N)$ such that:
\begin{itemize}
\item $\eta$ is fibrewise nondegenerate, that is, for every $p \in M$, $\eta$ is nondegenerate in $\mathrm{ker}(\phi_p)$,
\item The form $\omega = \eta + f^*(\rho \omega_{\Sigma})$ is symplectic with zero elliptic residue on $\mathcal {A}_{|D|}$ for every $\rho\in C^\infty (N)$ as long as  $\rho \geq \rho_0$.
\end{itemize}
\end{theorem}
Apart from the theorem above, \cite{CK18} also includes a general Gompf--Thurston result for Lie algebroid Lefschetz fibrations: under similar conditions on a Lie algebroid Lefschetz fibration one can construct a Lie algebroid symplectic form on the domain by adding a form which is symplectic on the fibres to a large multiple of the pull back of a symplectic form on the base.
 
Neither result can be directly applied to our case: Theorem~\ref{th:wishful} does not work because our divisor is not smooth, while the failure of the general result on Lie algebroid fibrations to yield stable generalized complex structures can already be seen in the simplest example.
\begin{example}
Consider the boundary fibration:
\begin{align*}
	f_1\colon (\cc^2, D) \to (\rr^2,Z), \qquad f_1(z_1,z_2) = (|z_1|^2,|z_2|^2),
\end{align*}
where $D$ and $Z$ are the coodinate axes on $\cc^2$ and $\rr^2$ respectively, as in Example~\ref{ex:basic}.

We can endow $\rr^2$ with the log-symplectic structure $d\log x_1 \wedge d\log x_2$, and consider on $\cc^2$ the closed elliptic form
\[\eta = -d\theta_1 \wedge d\theta_2 + d\log r_1 \wedge d \theta_2 + d\log r_2 \wedge d\theta_1,\]
which is non-degenerate on the fibres of $f_1$. The Gompf--Thurston theorem then provides us with a $1$-parameter family of forms
\begin{align*}
	\omega_t = \eta+  tf^*(d\log x_1 \wedge d\log x_2),
\end{align*}
which is symplectic for $t>1$. This poses a problem: although this defines a legitimate elliptic symplectic form, there is no value of $t$ for which it corresponds to a stable generalized complex structure, since $\abs{\Res_{r_1r_2}\omega_t} \neq \abs{\Res_{\theta_1\theta_2}\omega_t}$ for $t>1$. We conclude that the process of scaling up the symplectic structure on the base to achieve nondegeneracy is incompatible with the residue conditions. 
\end{example}
What we do next is to adapt Theorem~\ref{th:wishful} for the self-crossing case.
\begin{theorem}\label{th:main}
Let $f\colon(M^4,D^2)\to (N^2,Z=\partial N)$ be a homologically essential boundary Lefschetz fibration with connected fibres between compact connected oriented manifolds. Denote by $\phi\colon \mathcal{A}_{|D|} \to \mathcal{A}_Z$ the induced map of Lie algebroids. Then $(M,I_{|D|})$ admits an elliptic symplectic structure with zero elliptic residue and imaginary parameter which is compatible with $f$.

If $D$ is co-orientable and the index of each connected component of $D$ is $1$ this elliptic symplectic structure induces a stable generalized complex structure.
\end{theorem}
\begin{proof}
Fix a log-symplectic structure $\omega_N \in \Omega^2(N,\mathcal{A}_{Z})$. First we consider $f\colon M\backslash D[2] \to N\backslash Z[2]$. This is a homologically essential, proper, boundary Lefschetz fibration with smooth elliptic divisor, hence, by Theorem~\ref{th:wishful}, there is a form $\eta \in \Omega_{0,0}^2(M\backslash D[2];\log |D\backslash D[2]|)$ and a function $\rho_0 \in \Omega^0(N \backslash Z[2])$ (recall, Definition~\ref{def:zeroresform}) such that $\omega = \eta + f^*( \rho \omega_N)$ is a zero elliptic residue symplectic form for any function $\rho\in \Omega^0(N\backslash Z[2])$ with $\rho\geq \rho_0$.

Now we show how to change this construction so that the form it yields extends over $D[2]$, is elliptic symplectic with zero elliptic residue, and has imaginary parameter.

For each point $p \in D[2]$, fix open neighbouhoods $U_1 \Subset U_2 \Subset U_3$ and oriented coordinates charts defined on $U_3$  and $f(U_3)$ in which $f$ has the form
\[f(z_1,z_2) = (|z_1|^2, |z_2|^2).\]
As usual, we express the complex coordinates in $U_3$ in polar form, $z_i = r_i e^{i \theta_i}$, and denote by $(x_1,x_2)$ the coordinates on the base, so $f^*x_i = r_i^2$.

The strategy will be to change the symplectic form $\omega$ described above in a very precise way:
\begin{itemize}
\item in the complement of $U_3$, $\omega$ remains unchanged except for a further constant scaling of the symplectic form $\omega_N$,
\item in $U_3\backslash U_2$ we change $\eta$ into a multiple of $d\theta_2 \wedge d\theta_1$ and we  preserve nondegeneracy by rescaling the symplectic form $\omega_N$ by a constant,
\item in $U_2\backslash U_1$ we interpolate the possibly large $f^*\omega_N$ to $f^*(d\log x_1 \wedge d\log x_2)$ and observe that this interpolation does not spoil the symplectic condition,
\item in $U_1$ we extend the symplectic form as $d\theta_2 \wedge d\theta_1 + d\log r_1 \wedge d\log r_2$, which clearly has the desired properties at $p$.
\end{itemize}
Now we carry out his plan explicitly. Fix $\rho \geq \rho_0$. On $U_3$ we have by Lemma \ref{lem:zerorescohom} that
\[[\eta] \in H^2_{0,0}(U_3\backslash D[2],\mathcal{A}_{|D\backslash D[2]|}) \cong H^2(U_3\backslash D) =\rr,\]
and the generator of this cohomology pairs nonzero with the torus given in coordinates by $F = f^{-1}(r_1,r_2)$, where $r_1$ and $r_2$ are any two small positive numbers. Let $\lambda = \int_F \eta$ where integration is with respect to the fibre orientation of $F$, hence $\lambda >0$. On $U_3$ consider the elliptic form $\tilde\eta = \frac{\lambda}{4\pi^2}d\theta_2 \wedge d\theta_1$. Then $\tilde\eta$ is closed in $U_3$ and also integrates to $\lambda$ over $F$. Therefore $[\eta] =[\tilde\eta] \in  H^2(U_3\backslash\{p\},\mathcal{A}_{D\backslash\{p\}})$ and there is a one-form $\alpha \in \Omega^1(U_3\backslash\{p\}, \mathcal{A}_{|D\backslash\{p\}|})$ such that $\tilde\eta = \eta + d\alpha$.

Let $k\geq 1$, let $\psi_1$ and $\psi_2$ be positive functions on $f(U_3)$ such that $\psi_1$ is equal to $1$ in neighbourhood of $f(U_1)$ and has support in $f(U_2)$ and $\psi_2$ is equal to $1$ in neighbourhood of $f(U_2)$ and has support in $f(U_3)$. Then consider the form
\[\tilde\omega :=
\begin{cases}
			\frac{\lambda}{4\pi^2}(d\theta_2 \wedge d\theta_1 + d\log r_1 \wedge d\log r_2) & \mbox{ in } U_1,\\
			\tilde\eta +   f^*((1-\psi_1) k\rho \omega_N + \psi_1\frac{\lambda}{4\pi^2}d\log x_1 \wedge d\log x_2)& \mbox{ in } U_2\backslash U_1,\\
			\eta + d ((f^*\psi_2)\alpha) + f^*(k\rho \omega_N)& \mbox{ in } U_3\backslash U_2,\\
			\eta + f^*(k\rho \omega_N)& \mbox{ in } M\backslash U_3.
\end{cases}
\]
Because of our choice of bump functions, this form is smooth. Also, it is clearly closed. Since $k\geq 1$, we have $k\rho \geq \rho \geq \rho_0$ and hence $\tilde\omega$ is symplectic in $M\backslash U_3$ for all possible values of $k$.

On $\overline{U_3\backslash U_2}$ we observe that the form $\eta + d ((f^*\psi)\alpha)$ is fibrewise symplectic. Indeed, its restriction to each fibre it is given by
\[\eta +  (f^*\psi_2) d\alpha = (f^*\psi_2)(\eta + d\alpha) + (1-(f^*\psi_2))\eta = (f^*\psi_2)\tilde\eta + (1-(f^*\psi_2))\eta,\]
hence it is a convex combination of $\eta$ and $\tilde\eta$ and these are both symplectic and determine the same orientation on each fibre. Since $\eta + d ((f^*\psi)\alpha)$ is fibrewise symplectic and $\rho \omega_N$ is symplectic on $N$, the combination $\eta + d ((f^*\psi)\alpha) + f^*(k\rho \omega_N)$ is symplectic on the compact set $\overline{U_3\backslash U_2}$ as long as $k$ is large enough.

On $U_2 \backslash U_1$, the form $\tilde \eta$ is given by $\frac{\lambda}{4\pi^2} d\theta_2 \wedge d\theta_1$, while the summand $f^*((1-\psi_1) k\rho \omega_N + \psi_1d\log x_1 \wedge d\log x_2)$ is a convex combination of two log-symplectic structures on $N$ which determine the same orientation, that is
\[f^*((1-\psi_1) k\rho \omega_N + \psi_1\frac{\lambda}{4\pi^2} d\log x_1 \wedge d\log x_2) = f^*(\kappa d\log x_1 \wedge d\log x_2),\]
for some positive function $\kappa$ and hence on $U_2 \backslash U_1$
\[\tilde\omega= \frac{\lambda}{4\pi^2} d\theta_2 \wedge d\theta_1 + (f^*\kappa) d\log r_1 \wedge d\log r_2,\]
which is clearly (zero residue) elliptic symplectic.
 
Finally, on $U_1$ we have $\omega = \mathrm{Im}(i \frac{\lambda}{4\pi^2}d\log z_1 \wedge d\log z_2)$, showing that it has the desired properties.
\end{proof}
%
\section{Connected sums of boundary Lefschetz fibrations}\label{sec:connsum}
In this section we describe a connected sum procedure for boundary Lefschetz fibrations along zero-dimensional strata of their elliptic divisors. This procedure will allow us to construct elaborate examples out of basic ones. For simplicity, we immediately restrict ourselves to dimension four, but we note that since the connected sum takes place at points of the divisor, this procedure can also be carried out for boundary fibrations in higher dimensions.

Before we start taking connected sums of boundary Lefschetz fibrations, first recall from \cite[Lemma 6.1]{CKW20} that we can take connected sums of elliptic divisors.
\begin{lemma}\label{lem:divglue}
Let $M_1^4,M_2^4$ be oriented manifolds endowed with elliptic divisors $I_{\abs{D_1}},I_{\abs{D_2}}$ respectively, and let $p_i \in D_i[2]$. Then $M_1 \#_{p_1,p_2} M_2$ admits an elliptic divisor $I_{\abs{\tilde{D}}}$ for which the natural inclusions $(M\backslash \set{p_i},I_{\abs{D_i}}) \rightarrow (M_1\#_{p_1,p_2} M_2,I_{\abs{\tilde{D}}})$ are morphisms of divisors.
\end{lemma}
Similarly we can perform a self-connected sum, which when $M$ is connected corresponds to attaching a 1-handle and hence the diffeomorphism type of the resulting space is $M \#(S^1\times S^3)$.
\begin{lemma}
Let $M^4$ be an oriented connected manifold endowed with an elliptic divisor $I_{\abs{D}}$, and let $p_1,p_2 \in D[2]$ be distinct points. Then $M \#(S^1\times S^3)$ and admits an elliptic divisor $I_{\abs{\tilde{D}}}$ for which the natural inclusion $(M\backslash \set{p_1,p_2},I_{\abs{D}}) \rightarrow (M\#(S^1 \times S^3),I_{\abs{\tilde{D}}})$ is a morphism of divisors.
\end{lemma}
In this connected sum procedure the map $\Phi(z_1,z_2) = \frac{1}{\abs{z_1}^2+\abs{z_2}^2}(z_2,\bar{z}_1)$ is used to identify annuli. Here $(z_1,z_2)$ are local complex coordinates compatible with the orientation on the manifolds. There is some freedom in the constructions above. Given a choice of local coordinates $(z_1,z_2)$ around $p_1$ and $p_2$, we can compose the map $\Phi$ by a permutation of the coordinates. This does not change the topology of $M_1 \#_{p_1,p_2} M_2$ but it could change the topology of the zero locus of the divisor. In dimension four, because of this freedom in the ordering, there are potentially two different topological types of the zero locus of the divisor. Notice however that our notation does not reflect this ambiguity.
\begin{remark}[Connected components]\label{rem:concompcon}
Although there is some freedom in the choices we can still distinguish the number of connected components on the divisor on the connected sum:
\begin{enumerate}
\item When $p_1$ and $p_2$ lie in different connected components of the divisor, be that either in the connected sum of two manifolds or in a self-connected sum, the connected components containing $p_1$ and $p_2$ will combine into a single connected component of $\tilde{D}$.
\item When $p_1$ and $p_2$ are in the same connected component, $D$, a case that can only happen in a self-connected sum, the resulting divisor, $\tilde{D} \subset M\#(S^1\times S^3)$, may have one or two connected components originating from $D$.\qedhere
\end{enumerate}
\end{remark}
Next we show that the connected sum operation is also compatible with boundary (Lefschetz) fibrations. To describe how the connected sum procedure interacts with the base of the fibration we first consider what happens in the local model:
\begin{lemma}\label{lem:locfibglue} Let $\Delta_r \subseteq \rr^2$ be the triangle bounded by the axes and the line $x + y = r$, and let $(x,y)$ be oriented coordinates on $\Delta_r$ and $(z_1,z_2)$ be complex coordinates on $\D^4_r$, the disc of radius $r$. Consider the following maps:
\begin{itemize}
	\item $p\colon (\D^{4}_2\backslash \D^4_{1/2}) \to (\Delta_2\backslash \Delta_{1/2})$, given by $(z_1,z_2) \mapsto (\abs{z_1}^2,\abs{z_2}^2)$;
	\item $\Phi\colon (\D^{4}_2\backslash \D^4_{1/2}) \rightarrow (\D^{4}_2\backslash \D^4_{1/2})$, given by $(z_1,z_2) \mapsto \frac{1}{\abs{z_1}^2+\abs{z_2}^2}(z_2,\bar{z}_1)$;
	\item $\Psi\colon (\Delta_2\backslash \Delta_{1/2}) \rightarrow (\Delta_2\backslash \Delta_{1/2})$ given by $(x,y) \mapsto \frac{(y,x)}{(x+y)^2}$.
\end{itemize}
Then the following diagram commutes:
\begin{center}
\begin{tikzcd}
(\D^{4}_2\backslash \D^4_{1/2}) \ar[r,"\Phi"] \ar[d,"p"] & (\D^{4}_2\backslash \D^4_{1/2}) \ar[d,"p"] \\
(\Delta_2\backslash \Delta_{1/2}) \ar[r,"\Psi"] & (\Delta_2\backslash \Delta_{1/2}).
\end{tikzcd}
\end{center}
\end{lemma}
The proof of this lemma is a simple verification. Just as we used the map $\Phi$ to perform a connected sum compatible with elliptic divisors, we want to use the map $\Psi$ to define a sort of connected sum operation of the base:
\begin{definition}
Let $\Sigma_1,\Sigma_2$ be oriented surfaces with corners, and let $q_1,q_2$ be corners of $\Sigma_1,\Sigma_2$ respectively. The \textbf{oriented corner connected sum} of $\Sigma_1$ and $\Sigma_2$ is defined by identifying a trapezoid neighbourhood of $q_1$ to a trapezoid neighbourhood of $q_2$ via $\Psi$. The oriented corner connected sum is an oriented surface with corners denoted by $\Sigma_1 \#_{q_1,q_2} \Sigma_2$ (see Figure~\ref{fig:corner sum}).
\end{definition}
\begin{figure}[h]
\begin{overpic}[unit=1mm,scale=0.4]{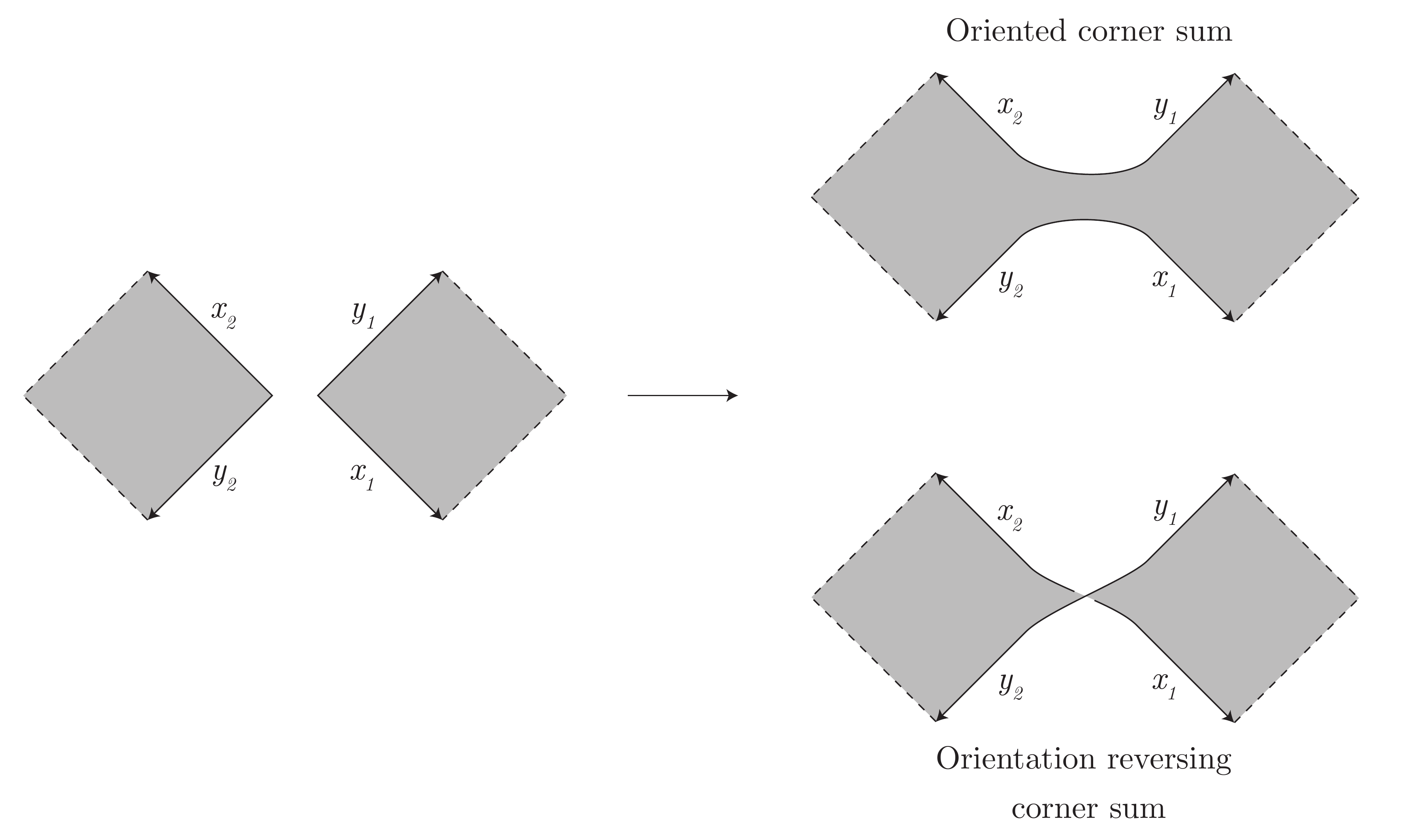}
\end{overpic}
\caption{Local oriented and orientation-reversing corner connected sums.}\label{fig:corner sum}
\end{figure}
The oriented corner connected sum is naturally oriented and does not depend on the neighbourhoods chosen. Together with the local normal form for fibrating boundary maps, we can now prove the following:
\begin{theorem}\label{thm:glue}
Let $f_i\colon(M_i^4,D_i) \rightarrow (N_i^2,\partial N_i)$ be boundary (Lefschetz) fibrations with connected fibres between oriented manifolds for $i=1,2$, let $p_i \in D_i[2]$ and $q_i = f(p_i)$. Then there exists a boundary (Lefschetz) fibration on one of the two possible connected sums $M_1 \#_{p_1,p_2} M_2$ whose base is the oriented corner sum $N_1 \#_{q_1,q_2}N_2$:
\begin{equation*}
	(f_1 \# f_2)\colon (M_1 \#_{p_1,p_2} M_2,\tilde{D}) \rightarrow (N_1 \#_{f(p_1),f(p_2)} N_2,\partial N_1\#_{q_1,q_2} \partial N_2), 
\end{equation*}
which is compatible with the (orientation-preserving) inclusions $M_i \backslash \set{p_i} \hookrightarrow M_1\#M_2$. 

Furthermore let $D_1',D_2'$ denote the connected components of the zero locus of the divisor $D$ containing $p_1,p_2$ respectively. Then the parities satisfy:
\begin{align*}
	\varepsilon_{\tilde{D}} = -\varepsilon_{D_1'}\varepsilon_{D_2'}.
\end{align*}
Finally $f_1 \# f_2$ is homologically essential if and only if $f_1$ and $f_2$ are.
\end{theorem}
\begin{proof}
By Lemma~\ref{lem:boundmaplocform} there exists neighbourhoods $U_1$, $U_2$ of $f(p_1)$, $f(p_2)$ respectively which provide coordinates as in the setting of Lemma~\ref{lem:locfibglue}. We perform the connected sum procedure using the maps described there. Because these maps are compatible with the fibrations on $M_1$ and $M_2$ we conclude that $M_1 \#_{p_1,p_2} M_2$ admits a boundary fibration. The computation of the parity is given in \cite[Theorem 6.7]{CKW20}.
\end{proof}
Recall that there are a priori two possible topological types for the elliptic divisor, depending on the ordering of the local coordinates. However, when we are presented with fibrations between oriented manifolds $f\colon (M_i,D_i) \to (N_i,\partial N_i)$ the orientation on the base determines an order for the strands of $D$ for every point $p_i\in D_i[2]$ (c.f.\ Remark~\ref{rem:there will be order}). The gluing of fibrations which is compatible with orientations on $M_i$ and $N_i$ is the one that flips the first and second strands arriving the points where the sum is performed. In particular, from the possible divisors discussed in Remark~\ref{rem:concompcon}, (2), only the one with two connected components occurs.
\begin{remark}[Non-orientable case]
If we were to allow the map $\Psi$ used in the corner sum to be orientation reversing, we would still be able to define a corner connected sum and obtain a boundary fibration. When taking the connected sum of two manifolds this does not cause a qualitative change in the outcome. However, if we use the orientation-reversing corner sum on the base for a self-connected sum, we see that the resulting base manifold is not orientable as a M\"obius band appears.
\end{remark}
Now that we understand precisely what happens to the connected components of the divisor on the self-connected sum, we can state the following:
\begin{corollary}\label{cor:selfglue}
Let $f\colon (M^4,D^2)\rightarrow (N^2,\partial N)$ be a boundary (Lefschetz) fibration with connected fibres between oriented manifolds, and let $p_1, p_2 \in D[2]$ be distinct. Then $M \# (S^1\times S^3)$ admits a boundary (Lefschetz) fibration $\tilde{f}$ which is compatible with the inclusion $M\backslash \set{p_1,p_2} \hookrightarrow M \#_{p_1,p_2} (S^1\times S^3)$, and for which $\tilde{D}[2] = D[2] \backslash \set{p_1,p_2}$.

Moreover let $D'_{p_1},D'_{p_2}$ denote the connected components of $D$ containing $p_1,p_2$ respectively.
\begin{itemize}
\item If $p_i \in D'_{p_i}[2]$  and $D'_{p_1} \neq D'_{p_2}$, then the corresponding connected component $\tilde{D}'$ of $\tilde{D}$ satisfies:
\begin{align*}
	\varepsilon_{\tilde{D}'} = -\varepsilon_{D_1'}\varepsilon_{D_2'};
\end{align*}
\item If $p_i \in D'_{p_i}[2]$ and $D'_{p_1} = D'_{p_2}$, then the corresponding connected components $\tilde{D}'_1,\tilde{D}'_2$ of $\tilde{D}$ satisfy:
\begin{align*}
\varepsilon_{\tilde{D}_1'}\varepsilon_{\tilde{D}_2'} = - \varepsilon_{D'_{p_1}} \varepsilon_{D'_{p_2}}.
\end{align*}
\end{itemize}
Finally $\tilde{f}$ is homologically essential if and only if $f$ is.
\end{corollary}
%
\section{Singularity trades}\label{sec:singtrades}
The goal of this section is to prove two theorems which allow one to trade Lefschetz for elliptic--elliptic singularities and vice-versa. To formulate these results, we need to recall the notion of vanishing cycle for both Lefschetz and elliptic singularities.

Given a boundary Lefschetz fibration $f \colon (M^4,D^2) \to (N^2,\partial N)$ and an elliptic or a Lefschetz singularity $p_1 \in M$, let $q_1  = f(p_1)$ be the corresponding singular value. We fix $q\in N$, a reference regular point of $f$ and $\gamma\colon [0,1] \to N$, a simple path connecting $q$ to $q_1$ which goes through no critical values of $f$ except for $q_1$ at time 1. We can consider $F_{q} = f^{-1}(q)$, $F_{\gamma} = f^{-1}(\gamma([0,1]))$ and the natural inclusion $\iota\colon F_{q} \to F_\gamma$. Then $F_{q}$ is a two-torus and $H_1(F_\gamma)$ is one-dimensional:
\begin{itemize}
\item In the case of a Lefschetz singularity the inclusion $H_1(F_{q}) \to H_1(F_\gamma)$ has kernel given by the Lefschetz vanishing cycle, which corresponds to the boundary of a Lefschetz thimble emanating from the singularity.
\item In the case of an elliptic singularity $F_\gamma$ is the product of  circle and a solid torus with $F_{q}$ as boundary, hence $\iota_* \colon H_1(F_{q}) \to H_1(F_\gamma)$ also has one-dimensional kernel given by the cycle in $F_{q}$ which becomes a boundary in $F_{\gamma}$.
\end{itemize}
In both cases the kernel of $\iota_*$ is generated by one primitive element in $H_1(F_{q};\mathbb{Z})$ which depends only on the homotopy class of $\gamma$ in $N\backslash\text{Crit}(f)$.
\begin{definition}
In the situation above, the \textbf{vanishing cycle} associated to the singular value $q_1$ and the homotopy class of the path $\gamma$ is either of the primitive elements in $H_1(F_{q};\zz)$ which generates the kernel of $H_1(F_{q};\zz) \to H_1(F_\gamma;\zz)$.
\end{definition}
\begin{definition}
Let  $f\colon (M^4,D^2) \to (N^2,\partial N)$ be a boundary Lefschetz fibration, let  $F_{q_0}$ and $F_{q_1}$  be Lefschetz or elliptic fibres. We say that the vanishing cycles at  $F_{q_0}$ and $F_{q_1}$ are a \textbf{dual pair} if
there is a simple path $\gamma\colon [0,1] \to N$ such that:
\begin{itemize}
\item $\gamma(0) = q_0$ and $\gamma(1)= q_1$,
\item $\gamma((0,1))$ only contains regular values of $f$,
\item the vanishing cycles on both ends of $\gamma$ together generate the integral homology of the regular torus fibre, say $F_{\gamma(1/2)}$.\qedhere
\end{itemize}
\end{definition}
With these notions at hand, we can give the precise statements of our singularity trade theorems.
\begin{theorem}[Elliptic--elliptic trade]\label{th:smoothing}
Let $f\colon (M^4,D) \rightarrow (N^2,\partial N)$ be a boundary Lefschetz fibration with connected fibres, and let $p \in D[2]$. Then $M$ admits a boundary Lefschetz fibration $\tilde f\colon (M^4,\tilde D) \rightarrow (\tilde N^2,\partial \tilde N )$ such that:
\begin{itemize}
\item $\tilde N$ is obtained from $N$ by smoothing out the corner $f(p)$,
\item $\tilde f$ and $\tilde{D}$ agree with $f$ and $D$ outside a small ball centered at $p$,
\item $\tilde D[2] = D[2] \backslash \{p\}$, i.e.~$\tilde f$ has one elliptic--elliptic singularity less than $f$,
\item $\tilde{D}$ and $D$ have the same parity,
\item $\tilde f$ has one Lefschetz singularity more than $f$,
\item $\tilde f$ has an elliptic singularity whose vanishing cycle forms a dual pair with the new Lefschetz vanishing cycle,
\item $\tilde{f}$ is homologically essential if and only if $f$ is.
\end{itemize}
By induction, any manifold which admits a boundary Lefschetz fibration admits one with a smooth embedded divisor.
\end{theorem}
The converse trade is given by the next theorem.
\begin{theorem}[Lefschetz trade]\label{th:converse trade}
\label{th:singularizing}Let $\tilde f\colon (\tilde{M}^4,\tilde D) \rightarrow (\tilde N^2,\partial \tilde N)$ be a boundary Lefschetz fibration with connected fibres, and assume that the vanishing cycles at a Lefschetz fibre, $F_{q_0}$, and at  an elliptic fibre, $F_{q_1}$ form a dual pair. Then there is a boundary Lefschetz fibration,  $f\colon (M^4,D) \rightarrow (N^2,\partial N)$, such that:
\begin{itemize}
\item $N$ is obtained from $\tilde N$ by adding a corner at $q_1$,
\item $f$ and $D$ agree with $\tilde f$  and $\tilde{D}$ outside $\tilde f^{-1}(V_2)$, where $V_2$ is a neighbourhood of the path that expresses a vanishing cycles as a dual pair,
\item $D[2] = \tilde D[2] \cup \{p\}$ and hence  $f$ has one elliptic--elliptic singularity more than $\tilde f$,
\item $D$ and $\tilde D$ have the same parity,
\item $f$ has one Lefschetz singularity less than $\tilde f$,
\item $f$ is homologically essential if and only if $\tilde{f}$ is.
\end{itemize}
\end{theorem}
The proofs of these theorems rely on the existence of specific boundary Lefschetz fibrations on $S^4$ and on  the open disc $\mathbb{D}^4$.
\begin{lemma}\label{lem:fibons4}
There exists a homologically essential boundary Lefschetz fibration with connected  fibres, $f_{S^4}\colon (S^4, D^2) \to (N,\partial N)$, with the following properties (see Figure~\ref{fig:closedan}):
\begin{itemize}
\item $D[2]$ has only one point, which has index $-1$,
\item $N$ is the disk with one corner,
\item $f_{S^4}$ has only one Lefschetz singularity,
\item the vanishing cycles of the Lefschetz fibre and any elliptic fibre form a dual pair.
\end{itemize}
\end{lemma}
\begin{figure}
\includegraphics[scale=.25]{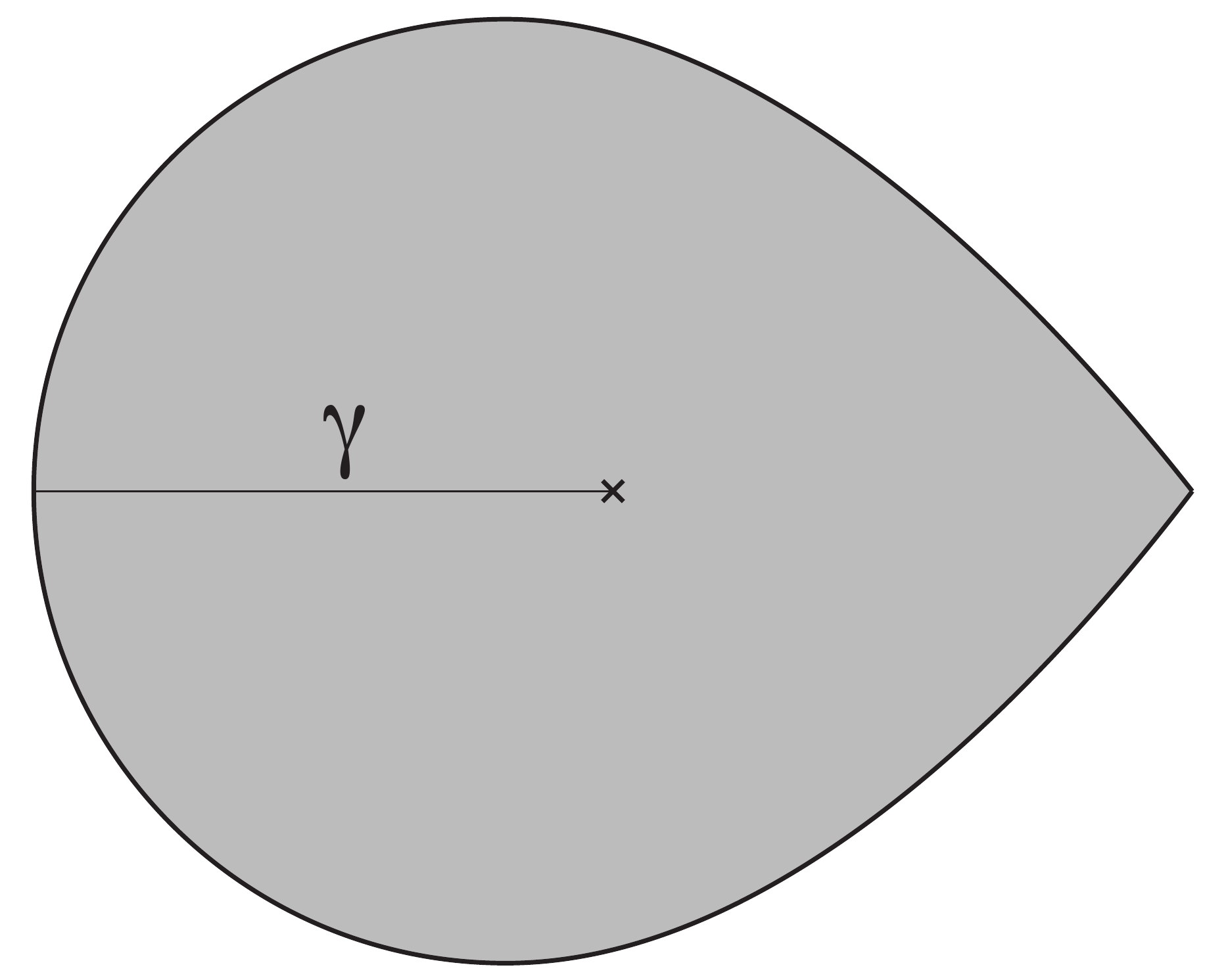}
\caption{The base of the boundary Lefschetz fibration on $S^4$ together with a path expressing the Lefschetz and elliptic singularities as a dual pair.}\label{fig:closedan}
\centering
\end{figure}
The proof of this lemma is somewhat long, so we will postpone it to this end of this section.
\begin{lemma}\label{lem:disc fibration}
Let $(\mathbb{D}^4,D)$ be the open disc in $\cc^2$ with divisor $I_{z_1}$ and let $\mathbb{D}^2_+ \subset \rr^2$ be the open half disc with boundary in the real axis
\[
\mathbb{D}^2_+  = \{ (x,y) \in \rr^2\colon x^2 + y^2 < 1 \mbox{ and } x\geq 0\}.
\]
Then there is a proper boundary Lefschetz fibration with connected fibres, $f_{\mathbb{D}^4} \colon (\mathbb{D}^4,D) \to (\mathbb{D}^2_+,\partial \mathbb{D}^2_+)$, such that:
\begin{itemize}
\item $f_{\mathbb{D}^4}$ has a single Lefschetz fibre,
\item the vanishing cycles of the Lefschetz fibre and the elliptic fibre form a dual pair.
\end{itemize}
Further, if $f \colon (M,D) \to (\mathbb{D}^2_+,\partial \mathbb{D}^2_+)$ is a proper boundary Lefschetz fibration with connected fibres with the two properties above, then $f$ is equivalent to $f_{\mathbb{D}^4}$, that is, there is a commutative diagram
\begin{center}
\begin{tikzcd}
\mathbb{D}^4 \ar[r] \ar[d,"f_{\mathbb{D}^4}"] & M \ar[d,"f"]\\
(\mathbb{D}^2_+,\partial \mathbb{D}^2_+) \ar[r,] & (\mathbb{D}^2_+,\partial \mathbb{D}^2_+)
\end{tikzcd}
\end{center}
where the horizontal maps are diffeomorphisms.
\end{lemma}
\begin{proof}
The existence of the fibration $f_{\mathbb{D}^4}$ follows from Lemma~\ref{lem:fibons4}. Indeed, we split the base of $f_{S^4}$ in two parts, $V_1$, a neighbourhood of the vertex and $V_2$, the rest of the base plus a small overlap with $V_1$, as indicated in Figure~\ref{fig:split}. Then, due to Lemma~\ref{lem:boundmaplocform}, on $f^{-1}(V_1)$, in appropriate coordinates, we have
\[V_1 = \{(x,y)\in \rr^2\colon x+y <1, x \geq0, y\geq 0\}\]
and the fibration is given by
\[f_{S^4}(z_1,z_2) = (|z_1|^2,|z_2|^2).\]
Hence, $f_{S^4}^{-1}(V_1)$ is a disc and its complement $f_{S_4}^{-1}(V_2)$ is also a disc. But
\[ f_{S_4}|_{V_2}\colon  V_2 \to f_{S^4}(V_2)\]
has all the properties required in the lemma after we choose a diffeomorphism between $V_2$ and $\mathbb{D}^2_+$. Therefore we have existence.
\begin{figure}
\begin{overpic}[unit=1mm,scale=0.25]{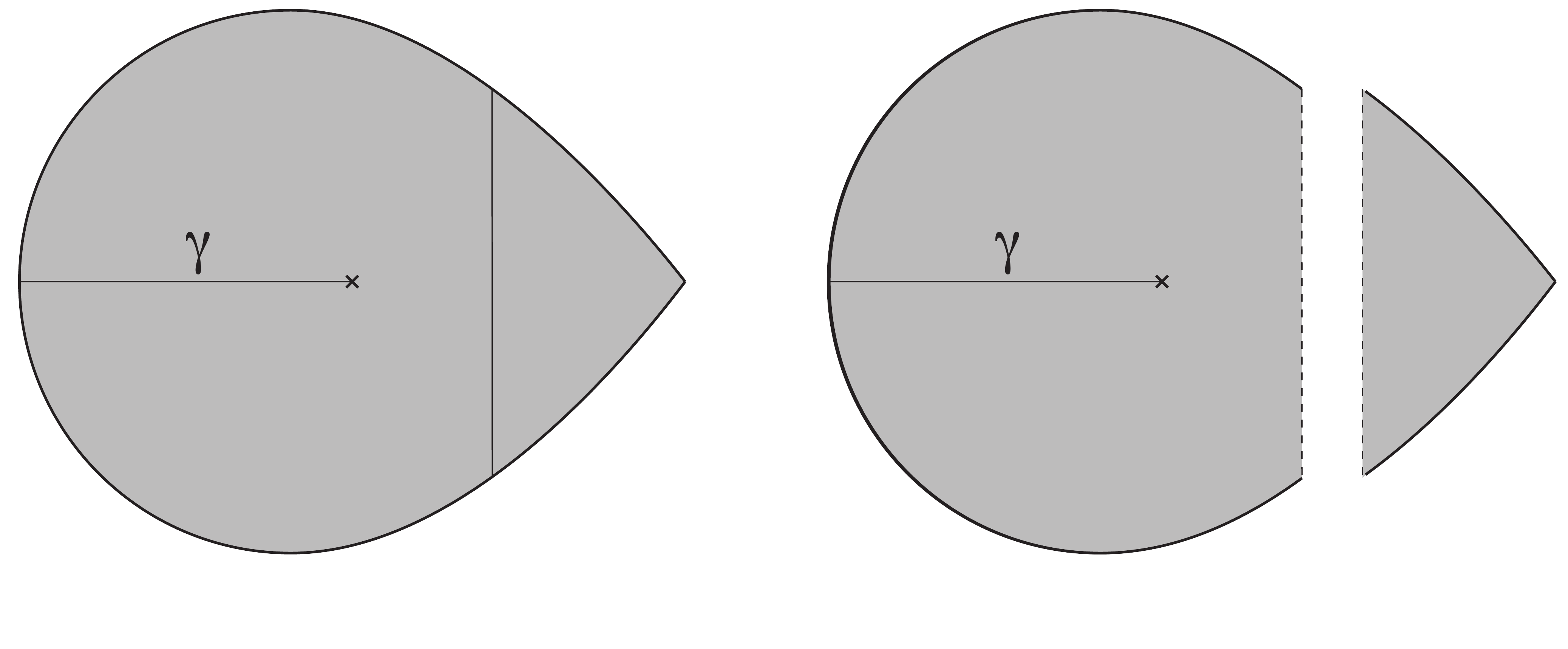}
	\put(20,0){$S^4$}
	\put(72,0){$\mathbb{D}^4$}
	\put(89,0){$\cup_{\partial}$}
        \put(98,0){$\mathbb{D}^4$}
        	\put(49,24){$\leadsto$}
        	\put(72,18){$V_2$}
        	\put(93,18){$V_1$}
\end{overpic}
\caption{The base of the boundary Lefschetz fibration on $S^4$ split in two halfs, each half being a fibration of $\mathbb{D}^4$.}\label{fig:split}
\centering
\end{figure}
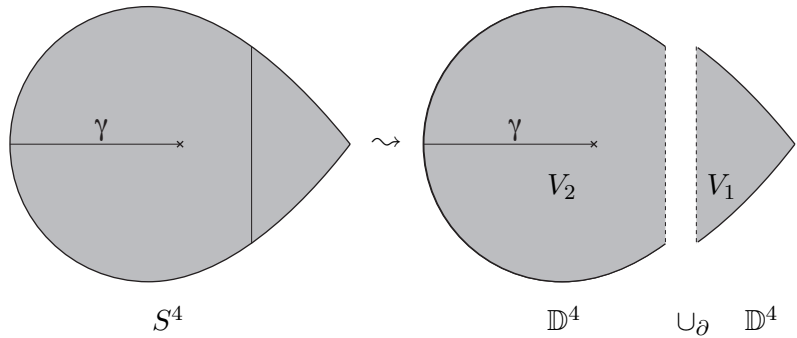

To prove the uniqueness part we study all possible ways such a fibration may arise. Let $f\colon M \rightarrow (\mathbb{D}^2_+,\partial \mathbb{D}^2_+)$ be a boundary Lefschetz fibration satisfying the assumptions of the lemma. Without loss of generality, we assume that the image of the Lefschetz singularity is $(2/3,0)$ and we split $\mathbb{D}^2_+$ in two parts:
\begin{align*}
	U_1 &= \{(x,y)\in \mathbb{D}^2_+\colon x \leq 1/2\},\\
	U_2 &= \{(x,y)\in \mathbb{D}^2_+\colon x \geq 1/2\}.
\end{align*}
The set $f^{-1}(U_2)$ is a neighbouhood of the Lefschetz fibre and hence its differentiable type as a fibration is fully determined \cite{GS99}.  Similarly, the set $f^{-1}(U_1)$ is a neighbourhood of an elliptic fibre hence its differentiable type as a fibration is also fully determined:
\[f^{-1}(U_1) = \mathbb{D}^2 \times S^1 \times (-1,1), \quad f(re^{i\theta}, \psi, t) = (r^2,t).\]
Therefore all possible different fibrations with the desired properties are determined by the different ways these two pieces can be glued together modulo the action of the isomorphism group of each half of the fibration.

Since the gluing takes place over a regular fibration over an interval, the isotopy class of the gluing map is determined by the isotopy class of the map it induces at a single fibre. Since the fibres are tori, this is in turn determined by the corresponding map in homology. Since the vanishing cycles form a dual pair there is, modulo the action of the isomorphism group of the fibration over $V_1$,  a unique way to glue these together.
\end{proof}
Next we show how to use Lemmas \ref{lem:fibons4} and \ref{lem:disc fibration} to prove  both singularity trade theorems:
\begin{proof}[Proof of Theorem~\ref{th:smoothing}]
Applying Theorem~\ref{thm:glue} to the boundary Lefschetz fibration on $M$ and on $S^4$ gives rise to a boundary Lefschetz fibration on $M\# S^4\simeq M$,  for which the inclusion $M\backslash \set{p} \hookrightarrow M$ preserves fibrations, in particular, we see that the new fibration on $M$ only changes in the small ball around $p$ used for the connected sum procedure. Since the divisor in $S^4$ has only one point in the top stratum, the new divisor satisfies $\tilde D[2] =  D[2] \backslash\{p\}$ and $\tilde{D}$ and $D$ have the same index. Given the way the fibrations are glued, we see that the effect on the base is to smooth out the corner corresponding to $f(p)$. Continuing inductively gives rise to a boundary Lefschetz fibration with embedded divisor.
\end{proof}

\begin{proof}[Proof of Theorem~\ref{th:converse trade}] 
Under the conditions of the theorem, $\gamma$ has a neighbourhood, $V_2$, diffeomorphic to $\mathbb{D}_+^2$ in which the fibration has only one Lefschetz singularity whose vanishing cycle forms a dual pair with the elliptic singularity. Hence, by Lemma~\ref{lem:disc fibration}, $f^{-1}(V_2)$ is diffeomorphic to $\mathbb{D}^4$ and $f$  is equivalent to the fibration of Lemma~\ref{lem:disc fibration}. Since the fibration on $S^4$ splits as two discs, one fibreing over $\mathbb{D}_+^2$ and the other fibreing over a neighbourhood, $V_1$, of the origin in $(\rr_+)^2$ (see Figure~\ref{fig:split}), we can realise $M \# S^4$ as follows: remove the disc $f^{-1}(V_2)$ and glue back, by the natural identification of the boundary, $f_{S^4}^{-1}(V_1)$.

Since this procedure corresponds to performing connected sum with $S^4$, the final manifold is still diffeomorphic to $M$ and the fibration only changes in the part that has been surgered in, which includes the removal of the Lefschetz singularity from $f^{-1}(V_2)$ and the inclusion of the elliptic--elliptic singularity of $f_{S^4}^{-1}(V_1)$. Finally, notice that the process of  filling the boundary of $f^{-1}(V_2)$ with $f_{S^4}^{-1}(V_1)$ is not compatible with the given orientations of these spaces since they both appear at opposite sides of a boundary in $S^4$. That is, the orientation of $M$ is compatible with the opposite orientation of  $f_{S^4}^{-1}(V_1)$.  Since the elliptic--elliptic singularity for the fibration in $S^4$ had index $-1$ and the orientation of $S^4$ was reversed in the connected sum process, the intersection index of the new elliptic--elliptic singularity on $M$ is $+1$ and hence the overall parity of the divisor is unchanged.
\end{proof}
To finish the proof of the trade theorems we must establish Lemma~\ref{lem:fibons4}, which we do next.
\begin{proof}[Proof of Lemma~\ref{lem:fibons4}] The proof is done in two steps. In the first step we show that if $M$ is the total space of a boundary Lefschetz fibration whose singularities are as stated in Lemma~\ref{lem:fibons4}, then $M = S^4$.  In the second step we show that such a fibration exists.

{\it Step 1.} We observe once again that $M$  is made of two fibrations glued together, as illustrated in Figure~\ref{fig:split}: one fibration with an elliptic--elliptic singularity over $V_1$ and one with a Lefschetz singularity over $V_2$. The fibration over $V_1$ is a copy of $\mathbb{D}^4$ added along its $S^3$ boundary, that is, $M = f^{-1}(V_2) \cup \mbox{ 4-handle}$. The space $f^{-1}(V_2)$ itself can be readily described as a handlebody: we start with a neighbourhood of a regular fibre, then add a $-1$-framed $2$-handle along the vanishing cycle of the Lefschetz singularity to obtain a neighbourhood of the Lefschetz singular fibre and a $0$-framed $2$-handle along the vanishing cycle of the elliptic singularity. Therefore the Kirby diagram of $M$ is the one depicted in Figure~\ref{fig:Kirby} (a). We can then slide the $2$-handle that goes around both 1-handles to obtain Figure \ref{fig:Kirby} (b) and see that the resulting $2$-handle separates as a $0$-framed $2$-handle from the rest of the diagram and hence cancels with the $3$-handle. The remaining pairs of $1$- and $2$-handles clearly cancel each other (Figure~\ref{fig:Kirby} (c)) leaving us with the empty diagram, which corresponds to $S^4$.
\begin{figure}
\includegraphics[scale=.25]{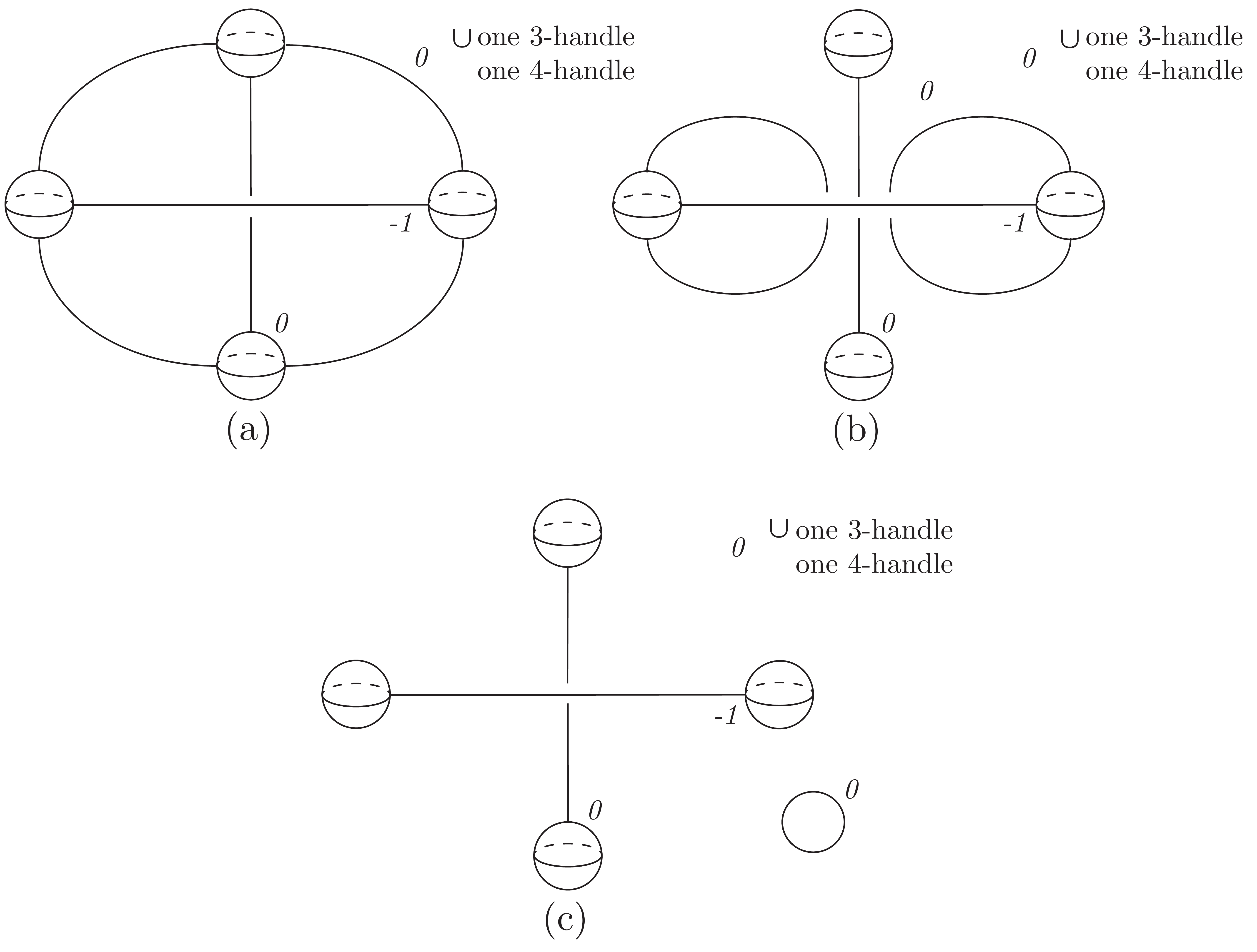}
\caption{Kirby diagram for the total space of the fibration described in Lemma~\ref{lem:disc fibration}.}\label{fig:Kirby}
\centering
\end{figure}

{\it Step 2.} To construct the desired fibration we will use a plumbing construction applied to the disc bundle of $\oo(2) \rightarrow \cc P^1$ in a way that is compatible with the natural torus fibration of that space. Throughout we will use fixed parametrizations $\phi_1,\phi_2\colon \cc^2 \to \oo(2)$ for which the change of coordinates is given by
\[\phi_2^{-1} \circ \phi_1(z,w) = (z^{-1},z^{-2}w).\]
We will refer to $\phi_1$ as parametrizing a trivialization of $\oo(2)$ with the south pole removed and similarly $\phi_2$ does not cover the fibre over the north pole.

Rotation on both coordinates in the parametrization $\phi_2$ give rise to a torus action on $\oo(2)$ which, in the parametrizations above, is given by
\begin{equation}\label{eq:torusaction}
\begin{aligned}
(e^{i \theta_1},e^{i\theta_2})\cdot \phi_1(z,w) &= \phi_1(e^{-i\theta_1}z,e^{i(-2\theta_1+\theta_2)}w),\\
(e^{i \theta_1},e^{i \theta_2})\cdot \phi_2(z,w) &= \phi_2(e^{i\theta_1}z,e^{i\theta_2}w).
\end{aligned}
\end{equation}
To describe the quotient of $\oo(2)$ by this torus action, we will also want to consider $[-1,1]\times \rr_+$. Of course this space can be parametrised by a single, rather obvious, chart, but it will be convenient to parametrise it by two charts instead. We consider the parametrizations
\[\psi_1 \colon \rr_+\times \rr_+ \to (-1,1] \times \rr_+ \subset   [-1,1] \times \rr_+ \qquad \psi_1(x_1,y_1) = \left(\frac{1-x_1}{1+ x_1},\frac{y_1}{(1+x_1)^2}\right),\]
\[\psi_2 \colon \rr_+\times \rr_+ \to [-1,1) \times \rr_+ \subset   [-1,1] \times \rr_+ \qquad \psi_2(x_2,y_2) = \left(-\frac{1-x_2}{1+ x_2},\frac{y_2}{(1+x_2)^2}\right),\]
and keep in mind that these parametrizations induce opposite orientations, with $\psi_2$ agreeing with the natural orientation of $[-1,1]\times \rr_+$.
\begin{lemma}\label{lem:pis}
If we let $h\colon S^2 \rightarrow \rr$ be the height function and $g\colon \mathrm{Sym}^2\oo(2) \rightarrow \rr$ be the Fubini-Study metric, then
\begin{align*}
f\colon \oo(2) &\rightarrow [-1,1]\times \rr_+\\
(z,w) &\mapsto (h(z),g_z(w,w)),
\end{align*}
defines a quotient map for the torus action on $\oo(2)$. Further, $f$ is a proper boundary fibration with elliptic divisor induced by the holomorphic log divisor consisting of the zero-section and fibres over the north and south pole.
\end{lemma}
\begin{proof}
In the parametrizations $\phi_i$, the height and distance function take the form:
\begin{align*}
h\circ \phi_1(z,w) = \frac{1-\abs{z}^2}{1+\abs{z}^2}, \quad
g\circ \phi_1(z,w) = \frac{\abs{w}^2}{(1+\abs{z}^2)^2},\\
h\circ \phi_2(z,w) = -\frac{1-\abs{z}^2}{1+\abs{z}^2}, \quad
g\circ \phi_2(z,w) = \frac{\abs{w}^2}{(1+\abs{z}^2)^2},
\end{align*}
which are clearly invariant under the $T^2$-action in Equation~\eqref{eq:torusaction}. Further, for $i=1,2$, the image of $f\circ \phi_i$ lands in the image of the parametrization $\psi_i$ and we can compute the expression for $f$ in these parametrizations:
\begin{equation}\label{eq:f in local coordinates} 
f_i(z,w) := \psi_i^{-1} \circ f \circ \phi_i(z,w) = (|z|^2,|w|^2),
\end{equation}
which shows clearly that $f$ not only is the quotient map but also a boundary fibration.
\end{proof}
Now we perform a plumbing on $\oo(2)$.
\begin{definition}
Let $\pi\colon M^{2n} \rightarrow N^n$ be a $\mathbb{D}^n$-bundle, and let $\mathbb{D}_1,\mathbb{D}_2$ be disjoint disks in $N$ over which $\pi$ is trivialisable. A \textbf{self-plumbing} of $\pi$ at $\mathbb{D}_1$ and $\mathbb{D}_2$ is obtained by identifying $\pi^{-1}(\mathbb{D}_1)\simeq \mathbb{D}_1\times \mathbb{D}^n$ and $\pi^{-1}(\mathbb{D}_2)\simeq \mathbb{D}_2 \times \mathbb{D}^n$ using a map which preserves the product structure but reverses the factors.
\end{definition}
For the case at hand, let $\mathbb{D}^2\oo(2)$ be the open $\varepsilon$-disk bundle with respect to the Fubini--Study metric. By restricting $f$ to $\mathbb{D}^2\oo(2)$, we obtain a proper boundary fibration $f\colon \mathbb{D}^2\oo(2) \rightarrow [-1,1]\times [0,\varepsilon)$. 

Further, we observe that $\phi_1$ and $\phi_2$ provide trivializations of $\mathbb{D}^2\oo(2)$, hence we can use them to perform a self-plumbing of $\mathbb{D}^2\oo(2)$ at the north and south poles. Let $M$ be defined as the self-plumbing of $\mathbb{D}^2\oo(2)$ via the trivializations $\phi_i$ and the map
\begin{align*}
	\Phi\colon \cc^2 \rightarrow \cc^2 : (z,w)&\mapsto (\bar{w},\bar{z}), 
\end{align*}
that is, $\phi_1(z,w)$ is identified with $\phi_2(\bar{w},\bar{z})$.

Since the map used for the plumbing preserves elliptic ideals and identifies the north and south pole, $M$ is endowed with an elliptic divisor with a single point in $D[2]$. Since the map $\Phi$ does not match co-orientations, the elliptic divisor in $M$ has intersection index $-1$. 

To endow $M$ with a boundary fibration we only need to take a quotient of the base, $[-1,1]\times [0,\varepsilon)$, by the equivalence relation that makes the following diagram commute:
\begin{center}
\begin{tikzcd}
\mathbb{D}^2\oo(2) \ar[r,"\sim_{\Phi}"] \ar[d,"f"] & \mathbb{D}^2 \oo(2) \ar[d,"f"]\\
\left[-1,1\right] \times [0,\varepsilon) \ar[r,"\sim{\Psi}"] & \left[-1,1\right]\times [0,\varepsilon)
\end{tikzcd}
\end{center}
Since $f$ is surjective, there is a unique identification, $\sim_\Psi$, that gives rise to such a diagram. In fact, we can easily compute it in the parametrizations $\psi_i$, where it is induced by the map $\Psi(x,y) = (1-y,x+1)$. That is, the point $\psi_1(x,y)$ is identified with the point $\psi_2(y,x)$. Since $\psi_1$ and $\psi_2$ induce opposite orientations, this identification preserves the natural orientation of $[-1,1]\times [0,\varepsilon)$ and the quotient is an oriented half-open cylinder with one corner (see Figure~\ref{fig:openan}).
\begin{figure}
\begin{overpic}[unit=1mm,scale=0.4]{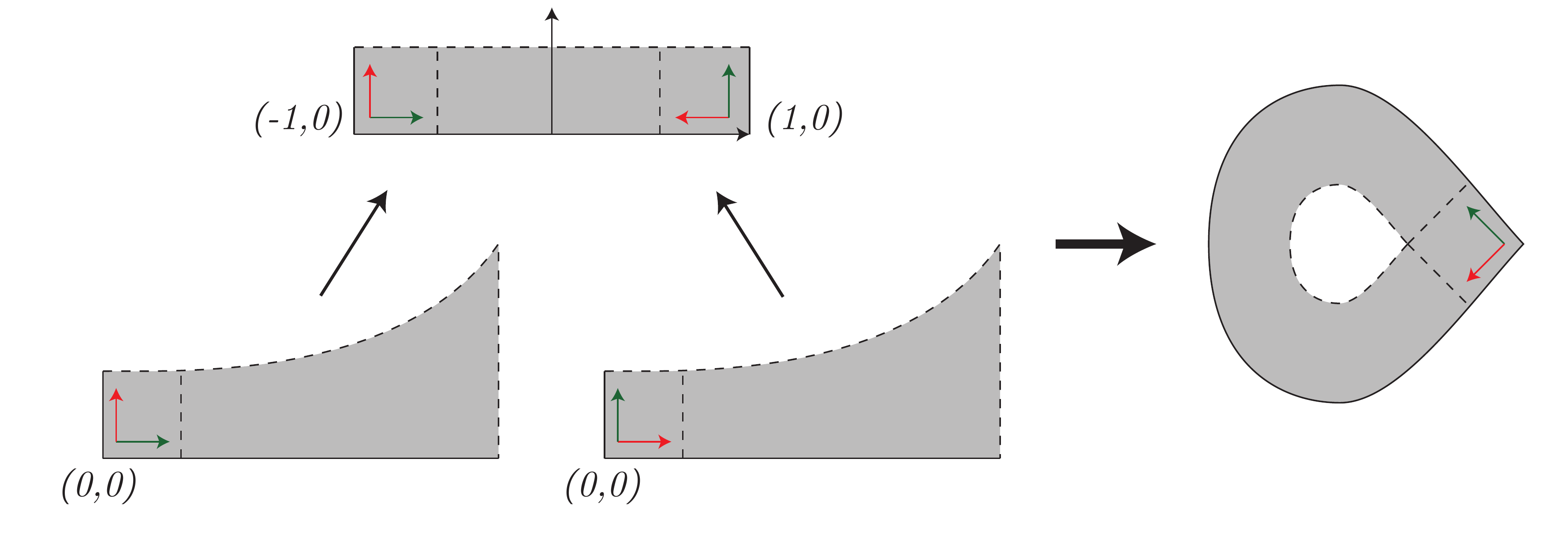}
	\put(26,26){$\psi_2$}
        \put(71,26){$\psi_1$}
        	\put(122,7){$N$}
\end{overpic}
\caption{The base of the boundary fibration constructed in Lemma~\ref{lem:plumbedfibr}.}\label{fig:openan}
\centering
\end{figure}
\begin{lemma}\label{lem:plumbedfibr}
The map $f\colon \mathbb{D}^2\oo(2) \rightarrow [-1,1]\times [0,\varepsilon)$ descends to a boundary fibration $\hat{f}\colon M \rightarrow N$.
\end{lemma}
Next we compute its monodromy along a generator of $\pi_1(N)$.
\begin{lemma}
Let $\hat{f}\colon M \rightarrow N$ be the boundary fibration from Lemma~\ref{lem:plumbedfibr}. Then the monodromy of $\hat{f}$ around a loop around the hole is a positive Dehn twist.
\end{lemma}
\begin{proof}
This is a direct computation using the given change of coordinates and the plumbing map $\Phi$. Indeed, all we need to to is to track what happens with the torus action as we move along from  the chart covered by $\phi_2$ to the chart covered by $\phi_1$ and then back to $\phi_2$ via $\Psi$:
\begin{align*}
(e^{i\theta_1}, e^{i\theta_2})\cdot\phi_2(z,w) &=\phi_2(e^{i\theta_1}z,e^{i\theta_2}w)   =\phi_1(e^{-i\theta_1}z^{-1},e^{i(-2\theta_1 + \theta_2)}z^{-2}w) \\
& \sim_\Phi \phi_2(e^{i(2\theta_1 - \theta_2)}\bar{z}^{-2}\bar{w},e^{i\theta_1}\bar{z}^{-1})  = (e^{i(2\theta_1 - \theta_2)},e^{i\theta_1})\phi_2(\bar{z}^{-2}\bar{w}, \bar{z}^{-1}).
\end{align*}
Therefore we see that, in the basis $\{e_{\theta_1},e_{\theta_2}\}$ for $H^1(F)$ corresponding to the generators of the action, the monodromy transformation is given by the matrix
\[
\begin{pmatrix}
 2 & -1\\
 1 &0
\end{pmatrix}
\]
Notice that using the complex orientation of $\oo(2)$ and the standard orientation of $\rr^2$, $\{e_{\theta_1},e_{\theta_2}\}$ is a negative basis for the homology of the fibre. Using this, we see that the transformation above is a positive Dehn twist on the cycle $e_{\theta_1}+ e_{\theta_2}$. 
\end{proof}
Now we can complete $M$ to a closed manifold by glueing a neighbourhood of a single Lefschetz fibre with vanishing cycle $e_{\theta_1} + e_{\theta_2}$ in the hole of the annulus. Finally we observe that  this vanishing cycle forms a dual pair with either of the two vanishing cycles of the elliptic singularity, which in the parametrization $\phi_2$ are given by either of the cycles $e_{\theta_1}$ or $e_{\theta_2}$.
\end{proof}
\begin{remark}
Simply drawing a base diagram for a boundary Lefschetz fibration does not guarantee the existence of a fibration that realises it. For example, there is no manifold whose base diagram is that of Figure~\ref{fig:closedan}, but for which the elliptic--elliptic singularity has intersection index $1$. In the construction above this would manifest itself in the fact that without using complex conjugation the monodromy of the plumbing would be a negative Dehn twist. This highlights that the long second step in the proof above is indeed necessary. 
\end{remark}
%
\section{Examples}\label{sec:examples}
In this section we give several concrete examples of boundary fibrations. We will first show that they arise naturally as the quotient maps of effective torus actions and that our framework fits particularly well with the theory of integrable systems. This connection provides us immediately with a wealth of examples of both boundary fibrations and stable generalized complex structures. We will further illustrate our constructions by showing how starting with simple examples (of manifolds with torus actions) we can use the connected sum procedure to obtain many more examples of boundary fibrations.
\subsection{Torus actions}
We show that quotient maps of torus actions provide boundary fibrations.
\begin{proposition}\label{prop:quotientmap}
Let $T^n$ act effectively on a smooth manifold $M^{2n}$, with connected isotropy groups. Then:
\begin{itemize}
\item $N:= M^{2n}/T^n$ is a manifold with corners;
\item the quotient map defines a boundary fibration $f\colon (M,D) \rightarrow (N,\partial N)$ with connected fibres;
\item the intersection stratification of the elliptic ideal coincides with the stratification by orbit types on $M$;
\item $ND[1]$ is co-orientable;
\item if $M$ is oriented, then so is $N$;
\item if $M$ is four-dimensional and the action is not free, $f$ is homologically essential.
\end{itemize}
\end{proposition}
\begin{proof}
Let $p \in M$, and let $G_p$ denote the isotropy group of $p$ and let $\mathcal{O}_p$ denote the orbit of $p$. By assumption, $G_p$ is connected and therefore isomomorphic to $T^\ell$ for some $\ell\leq n$. By the slice theorem, there exists a neighbourhood of $\oo_p$ which is equivariantly diffeomorphic to a neighbourhood of the zero section in
\begin{align*}
G \times_{G_p} N_p\mathcal{O}_p,
\end{align*}
where $G_p$ acts linearly on $N_p\mathcal{O}_p$ by the differentiated action. Because all groups in consideration are Abelian and connected, this implies that there is a neighbourhood $U$ around $p$ of the form
\begin{align*}
U = T^{n-\ell} \times (\rr^{n-\ell}\times \cc^\ell).
\end{align*}
The $T^n = (T^{n-\ell}\times T^\ell)$-action of $U$ decomposes as $T^{n-\ell}$ acting by multiplication on $T^{n-\ell}$ and $T^\ell$ acting linearly on $\cc^\ell$. Since the irreducible representations of $T^\ell$ are one-dimensional, we may without loss of generality assume that each coordinate line in $\cc^\ell$ is preserved by the action. Therefore, if we let $\mathfrak{t}$ denote the Lie algebra of $T^\ell$, let $\mathfrak{l}$ denote the kernel of $\exp\colon \mathfrak{t} \rightarrow T^\ell$, with minimal generating set $\{\xi_1,\dots,\xi_\ell\}$ and choose $\{\alpha_1,\ldots,\alpha_\ell\} \in \mathfrak{l}^*$ the dual basis for the dual lattice, then the action on each irreducible representation has the form
\begin{align*}
\exp(\Theta)\cdot z_j = e^{2\pi i \inp{\Theta,n_j\alpha_j}}z_j, \qquad \Theta \in \mathfrak{t}.
\end{align*}
Since the action is effective we have that $n_j \neq 0$, and because the isotropy groups are connected we must furthermore have $n_j = \pm 1$. Hence, after appropriately changing the signs of some of the $\alpha_j$, the $T^\ell$-action is given by
\begin{align*}
(\exp(\theta_1\xi_1\cdot \ldots \cdot \theta_\ell\xi_\ell))\cdot (z_1,\ldots,z_\ell) = (e^{2\pi \theta_1 i}z_1,\ldots e^{2\pi \theta_\ell i}z_\ell). 
\end{align*}
This normal form for the action has the following consequences:
\begin{itemize}
\item The quotient manifold is endowed with charts of the form $\rr^{n-\ell} \times (\cc^\ell)/T^\ell \simeq \rr^n_\ell$, and is therefore a manifold with corners;
\item The quotient map $f\colon M \rightarrow N$ in the above local coordinates is given by
\begin{align*}
f\colon T^{n-\ell} \times (\rr^{n-\ell}\times \cc^\ell) &\rightarrow \rr^n_\ell\\
(q,x,z_1,\ldots,z_\ell) &\mapsto (x,\abs{z_1}^2,\ldots,\abs{z_\ell}^2).
\end{align*}
By Lemma~\ref{lem:boundmaplocform} we see that $f$ is a boundary fibration with respect to the log divisor $\partial N$;
\item Because the vanishing locus of the induced elliptic ideal is given by $f^{-1}(\partial N)$ it follows that the intersection stratification coincides with the orbit type stratification;
\item At points $p \in D[1]$, the isotropy group is given by $S^1$ and therefore $N_p\mathcal{O}_p$ inherits an $S^1$-action and consequently admits an orientation. We conclude that $D[1]$ is co-orientable;
\item When $M$ is oriented, a choice of orientation for $T^n$ gives rise to an orientation for $N$ by observing that $M\backslash D \rightarrow N \backslash \partial N$ is a principal $T^n$-bundle;
\item When $M$ is four-dimensional and the action is not free it is shown in \cite{OR70} that $f$ admits a section. Therefore it follows that $f$ is homologically essential.\qedhere
\end{itemize}
\end{proof}
The group actions underlying toric manifolds satisfy the conditions of this proposition, leading to the following result:
\begin{corollary}
Let $(M^{2n},\omega)$ be a toric manifold and let $f\colon M^{2n} \rightarrow \Delta^n$ denote the quotient map. Then $f$ is a boundary fibration.
\end{corollary}
In four dimensions Proposition~\ref{prop:quotientmap} provides us with fibrations that satisfy nearly all the assumptions required to apply Theorem~\ref{th:main}. However, the torus action does not guarantee that the parity of the elliptic divisor is one. To proceed we must add hypotheses to ensure that this is the case.
\begin{proposition}\label{ex:toric}
Let $f\colon (M^4,\omega) \rightarrow \rr^2$ be a toric manifold. Then the parity of the elliptic divisor obtained from Proposition~\ref{prop:quotientmap} is $1$, and therefore $M$ admits a stable generalized complex structure compatible with $f$.
\end{proposition}
\begin{proof}
By Proposition~\ref{prop:quotientmap} we have that $f$ is a boundary fibration, and therefore by Theorem~\ref{th:main} the manifold $M$ admits an elliptic symplectic structure. As each of the preimages of the faces of the moment polytope is a symplectic submanifold of $(M,\omega)$, the symplectic structure provides each component of the elliptic divisor with a natural co-orientation for which the intersections have positive index. It follows that the parity of the elliptic divisor is $1$.
\end{proof}
\subsection{Simple examples}
We give examples of boundary fibrations obtained from torus actions which will serve as the building blocks for the connected sum procedure. The existence of stable generalized complex/elliptic symplectic structures in all examples below was known previously. The new input is that these spaces admit fibrations compatible with the structure in question.
\begin{example}[$\mathbb{C}P^2$]\label{ex:cp2}
Consider the standard toric structure on $\cc P^2$. Proposition~\ref{ex:toric} implies that $f$ is a homologically essential boundary fibration and that $\mathbb{C}P^2$ admits an elliptic divisor with parity $1$ (three lines intersecting at different points). Therefore $\cc P^2$ admits a stable generalized complex structure compatible with its moment map.
\end{example}

\begin{example}[$\bar{\mathbb{C}P}^2$]
We consider $\bar{\cc P}^2$, i.e. $\cc P^2$ with the orientation opposite to the standard complex structure. As an oriented manifold this is not a toric manifold, but there is still a $T^2$-action with connected isotropies present. Therefore \ref{prop:quotientmap} implies that the quotient map is a homologically essential boundary fibration. Consequently, by Theorem~\ref{th:main} there exists a compatible elliptic symplectic structure with imaginary parameter on $\bar{\cc P}^2$. The parity of the elliptic divisor is $-1$ so this symplectic structure does not induce a stable generalized complex structure. As $\bar{\cc P}^2$ is not almost complex it can not have a stable generalized complex structure, hence this problem can not be remedied.
\end{example}
\begin{example}[$S^2\times S^2$]
Let $(S^2\times S^2)$ be given its standard toric structure, i.e.~the symplectic form is the product of the standard area forms and $T^2$ acts on rotation by $S^1$ one each of the factors. Proposition~\ref{ex:toric} implies that the quotient map is a homologically essential boundary fibration and that $S^2\times S^2$ admits a compatible stable generalized complex structure.
\end{example}
\begin{example}[$S^4$]\label{ex:S4}
Consider $S^4 \subset \cc^2 \times \rr$ and let $T^2$ act in the standard way on $\cc^2$. This provides an effective $T^2$-action on $S^4$ with connected isotropies. Therefore by Proposition~\ref{prop:quotientmap} we find that the quotient map is a homologically essential boundary fibration. Consquently Theorem~\ref{th:main} implies the existence of a compatible elliptic symplectic structure with imaginary parameter on $S^4$. The parity of the divisor is $-1$. Just as $\overline{\cc P^2}$, $S^4$ is not almost-complex so the index can not be fixed by making different choices of divisor or orientations. 
\end{example}
The following example of a boundary fibration appears also in \cite{CK18}:
\begin{example}[$S^3\times S^1$]
There are two interesting $T^2$-actions on $S^3\times S^1$. First, consider $S^3 \subset \cc^2$ as the unit sphere and restrict the natural $T^2$-action on $\cc^2$ to $S^3$. This provides an effective $T^2$-action on $S^3$ with $S^1$ isotropy at all points in the intersection with the coordinate hyperplanes. Extending the $T^2$-action trivially to the $S^1$-factor provides an effective $T^2$-action on $S^3 \times S^1$ with only $S^1$ isotropy groups. The quotient map
\begin{align*}
f_1\colon (S^3\times S^1,D_1) \rightarrow (I\times S^1,\set{0,1}\times S^1),
\end{align*}
then becomes a homologically essential boundary fibration by Proposition~\ref{prop:quotientmap}. Note that $D_1$ is given by the union of two disjoint tori.

Another $T^2$-action on $S^3\times S^1$ is obtained by letting one $S^1$ act by rotation on one of the coordinates of $S^3 \subset \cc^2$ and let the other act by multiplication on $S^1$. The quotient map
\begin{align*}
f_2\colon (S^3\times S^1,D_2) \rightarrow (\mathbb{D}^2,\partial \mathbb{D}^2),
\end{align*}
then again becomes a homologically essential boundary fibration by Proposition~\ref{prop:quotientmap}. In this case $D_2$ is a single torus. In both cases Theorem~\ref{th:main} implies the existence of a compatible elliptic symplectic structure with zero elliptic residue. Moreover, as the vanishing locus of the elliptic divisor is smooth and co-orientable, we obtain two stable generalized complex structures on $S^3\times S^1$.
\end{example}
The example we consider next is more elaborate than the previous ones. The existence of stable generalized complex structures on these spaces is a consequence of the more general Theorem 2 from \cite{MR3177992}.
\begin{example}[$(\#nS^1\times S^2)\times S^1$]\label{ex:newones} 
In \cite{OR70}, it is shown that for $2g+h>1$, the manifold $M= (\#( 2g+h -1) S^1\times S^2)\times S^1$ admits an effective $T^2$-action with connected isotropy groups over a base, $B$, which is a surface of genus $g$ with $h$ small open discs removed. In fact, part of the action is just rotation of the last $S^1$-factor, so this action has no fixed points (a fact that also follows from the Euler characteristic of $M$ being $0$).

By Proposition~\ref{prop:quotientmap} we conclude that there exists a homologically essential boundary fibration
\begin{align*}
f\colon ((\#(2g+h-1)S^1\times S^2)\times S^1,D) \rightarrow (B,\partial B).
\end{align*}
The degeneracy locus consist of $h$ disjoint tori -- precisely the number of boundary components of $B$ -- and is in particular co-orientable. Consequently by Theorem~\ref{th:main} there exists a compatible (smooth) stable generalized complex structure on $M$ whose type change locus has $h$ connected components.
\end{example}
To illustrate the elliptic-elliptic trade theorem we give some examples:
\begin{example}[$\mathbb{C}P^2$]
Applying Theorem~\ref{th:smoothing} to Example~\ref{ex:cp2} yields several boundary Lefschetz fibrations $f\colon (\mathbb{C}P^2,D) \rightarrow (N, \partial N)$. The number of elliptic--elliptic and Lefschetz singularities adds up to three, but any combination is possible. See also Remark~\ref{rem:eulerchar}.
\end{example}
\begin{example}[$S^4$]
Applying Theorem~\ref{th:smoothing} to Example~\ref{ex:S4} yields a boundary Lefschetz fibration $f\colon (S^4,\tilde{D}) \rightarrow (\mathbb{D}^2,\partial \mathbb{D}^2)$ with two Lefschetz singularities. Because the parity of the original divisor on $S^4$ is $-1$, the new divisor $\tilde{D}$ will be non-co-orientable. Therefore it is non-orientable and as it admits an $S^1$-fibration it must then be a Klein bottle.
\end{example}
\subsection{Main class of examples}
Using the above examples as building blocks we can now construct many more examples:
\begin{theorem}\label{th:examples}
The manifolds in the following two families admit homologically essential boundary fibrations:
\begin{itemize}
	\item $X_{n,\ell} := \#n(S^2\times S^2)\#\ell(S^1\times S^3)$, with $n,\ell \in \mathbb{N}$;
	\item $Y_{n,m,\ell} := \#n \cc P^2 \#m \bar{\mathbb{C}P}^2\#\ell (S^1 \times S^3)$, with $n,m,\ell \in \mathbb{N}$,
\end{itemize}
whenever their Euler characteristic is non-negative. Therefore, each of these manifolds admits a compatible elliptic symplectic structure, which induces a stable generalized complex structure if $1-b_1 + b_2^+$ is even.
\end{theorem}
\begin{proof}
In the previous section we exhibited boundary fibrations on $\cc P^2,\bar{\mathbb{C}P}^2$ and $S^2\times S^2$ with $3,3,4$ points in $D[2]$ respectively. Therefore we may apply Theorem~\ref{thm:glue} inductively to obtain homologically essential boundary fibrations on $X_{n,0}$ and $Y_{n,m,0}$ for all possible values of $n$ and $m$, including $n=m=0$ by Example~\ref{ex:S4}. The number of points in $D[2]$ for these manifolds is $2n+2$ and $n+m+2$ respectively. Therefore we can apply Corollary~\ref{cor:selfglue} respectively $n+1$ and $\floor{\frac{n+m+2}{2}}$-times to obtain homologically essential boundary fibrations on $X_{n,l}$ and $Y_{n,m,\ell}$, for $\ell \leq n+1,\floor{\frac{n+m+2}{2}}$ respectively. A simple computation of the Euler characteristic of these manifolds shows that this is precisely when their Euler characteristic is non-negative. The parity of the divisor in $\cc P^2, \bar{\cc P}^2$ and $S^2\times S^2$ is $1,-1,1$ respectively. Therefore Theorem~\ref{thm:glue} gives us that the parity of $X_{n,0}$ and $Y_{n,m,0}$ is $(-1)^{n-1}$. Corollary~\ref{cor:selfglue} gives us that the parity of the divisor in $X_{n,\ell}$ and $Y_{n,m,\ell}$ is $(-1)^{n-1+\ell}$. By Theorem~\ref{th:main} these manifolds admit compatible elliptic symplectic structures. These induce stable generalized complex structures when $(-1)^{n-1+\ell}=1$, which is to say that $1-b_1 + b_2^+$ is even.
\end{proof}
The following remarks elaborate on the assumptions on the $n,m$ and $\ell$ in the above theorem.
\begin{remark}[Euler characteristic]\label{rem:eulerchar}
The condition on the Euler characteristic is necessary. Indeed a simple application of Mayer--Vietoris shows that if $f\colon M^4 \rightarrow \Sigma^2$ is a boundary Lefschetz fibration over a surface with $k$ corners, and $\ell$ Lefschetz singular fibres then $\chi(M) = k + \ell$. In particular, we find that the Euler characteristic of a manifold admitting a boundary Lefschetz fibration is necessarily non-negative. Therefore we conclude that we found all members of the families appearing in Theorem~\ref{th:examples} that admit boundary fibrations.

This observation has another consequence. One can perform connected sum of manifolds with elliptic divisors $(M_i,D_i)$ at points in $p_i \in D_i[1]$, but, differently from the case of points in $D_i[2]$, connected sum at $D_i[1]$ is not compatible with the existence of boundary Lefschetz fibrations. Indeed, if each $M_i$ had such a fibration which induced one in  $M_1 \#_{p_1,p_2} M_2$ by some identification of neighbourhoods of $p_1$ and $p_2$, the number of corners and Lefschetz singularities would be no less than the sum of corners and singularities for each $M_i$, but the Euler characteristic of $M_1\#M_2$ equals the sum of the Euler characteristics of $M_1$ and $M_2$ minus two.
\end{remark}
\begin{remark}[Betti numbers]
The existence of a generalized complex structure on a manifold implies the existence of an almost-complex structure. Such a structure cannot exist when $1-b_1 + b_2^+$ is odd, which explains that we found all members of the families appearing in Theorem~\ref{th:examples} that admit stable generalized complex structures arising from boundary fibrations.
\end{remark}
\begin{remark}[Torus actions]
Torus actions persist under taking connected sums of disjoint manifolds at fixed points \cite{OR70}. In fact, \cite{OR70} provides a classification of simply connected four-manifolds with effective torus actions and connected isotropy groups. The manifolds admitting such actions are precisely the manifolds $X_{n,0}$, $Y_{n,m,0}$, and $S^4$ appearing in Theorem~\ref{th:examples}. Whenever such a $T^2$-action is present it is possible to ensure that the elliptic symplectic structure arising from Theorem~\ref{th:examples} is $T^2$-invariant, hence we obtained all such simply connected four-manifolds admitting $T^2$-invariant stable generalized complex structures.

In the non-simply connected case, \cite{OR70} also provides a classification of effective non-free torus actions with only $S^1$-isotropy groups on compact oriented connected four-manifolds. It is proven that any of these manifolds is of the form as described in Example~\ref{ex:newones}, hence we have obtained all manifolds with such actions and $T^2$-invariant stable generalized complex structures.
\end{remark}
\subsection{Relation to semi-toric geometry}
We finish by relating our results to semi-toric geometry. Recall that a \textbf{focus--focus} singularity of a completely integrable system $(M,\omega,f)$ is a point $p \in M$ where there are Darboux coordinates $(x_1,y_1,x_2,y_2)$ for $\omega$ in which $f$ takes the form
\begin{align*}
(x_1,y_1,x_2,y_2) \stackrel{f}{\longmapsto} (x_1y_2-x_2y_1,x_1x_2+y_1y_2).
\end{align*}
\textbf{Semi-toric manifolds} (\cite{VP08}) are generalisations of four-dimensional toric manifolds where the moment map, besides elliptic and elliptic--elliptic singularities, may also have focus--focus singularities. If we use the above Darboux coordinates to define complex coordinates
\begin{align*}
(w_1,w_2) = \frac{1}{4}(x_1+y_2+i(x_1-y_2),x_1-y_2+i(x_1+y_2)),
\end{align*}
we see that the point $p$ becomes a Lefschetz singularity of the moment map $f$.
\begin{proposition} Moment maps of semi-toric manifolds are boundary Lefschetz fibrations. Consequently, semi-toric manifolds admit compatible stable generalized complex structures, for which the elliptic divisor is the pre-image of the boundary of the moment map image.
\end{proposition}
\begin{proof} In light of Theorem~\ref{th:main} and the toric case (Proposition~\ref{ex:toric}) we need only argue that the map is homologically essential. This follows because the homotopy type of $M\backslash D$ is obtained from a regular fibre by adding $2$-cells along the vanishing cycles corresponding to each Lefschetz singularity.
\end{proof}
\begin{remark}
Theorem~\ref{th:smoothing} trades an elliptic--elliptic singularity for a Lefschetz singularity in the context of a fibration without further geometric structures. This is reminiscent of the nodal trade/Hamiltonian Hopf bifurcation from semi-toric geometry \cite{Zung03,MR815106} in the context of Lagrangian fibrations. In the Hamiltonian Hopf bifurcation, elliptic--elliptic singularities are traded for focus-focus singularities, which by the above are equivalent to Lefschetz singularities. However these maps interact differently with the underlying geometric structure. Notably, in the semi-toric version, the base of the fibration has a singular integral affine structure which helps with the extension of the fibration beyond a neighbourhood of the singularities involved.

The converse trade for semi-toric geometry, similar to our Theorem~\ref{th:converse trade}, appeared in \cite{MR2670163}. There, the authors also make use of the singular integral affine structure structure present in such integrable systems.
\end{remark}
%
%
\bibliographystyle{hyperamsplain-nodash}
\bibliography{references} 
\end{document}